\documentclass[oneside,reqno,11pt]{amsart}

\usepackage{amsmath,amsfonts,amssymb,enumitem,mathtools, etoolbox}

\usepackage[T1]{fontenc} \usepackage{lmodern,times}
\usepackage{microtype} 

\usepackage[usenames,dvipsnames]{xcolor}
\usepackage{hyperref}
\hypersetup{colorlinks, citecolor=Violet, linkcolor=NavyBlue, urlcolor=Blue}

\usepackage[left=1in,right=1.0in,marginparwidth=1.7in]{geometry}
\allowdisplaybreaks

\makeatletter
\renewcommand\paragraph{\@startsection{paragraph}{4}{\z@}%
                                    {3.25ex \@plus1ex \@minus.2ex}%
                                    {-0.5em}%
                                    {\normalfont\normalsize\it}}
\makeatother
\usepackage{tabto} 
\usepackage{tikz-cd} 

\numberwithin{equation}{section}
\theoremstyle{plain}                
\newtheorem{theorem}{Theorem}[section]
\newtheorem{lemma}[theorem]{Lemma}
\newtheorem{proposition}[theorem]{Proposition}

\theoremstyle{definition}           
\newtheorem{definition}[theorem]{Definition}
\newtheorem{example}[theorem]{Example}

\theoremstyle{remark}
\newtheorem{remark}[theorem]{Remark}

\DeclareMathOperator\Lip{Lip}

\DeclareMathOperator\Id{Id}

\DeclareMathOperator\image{Im}

\DeclareMathOperator\Cl{Cl}
\DeclareMathOperator\ulim{\underline{\lim}}

\newcommand{\tot}{\tfrac{1}{2}} 
\newcommand{\oo}[1]{\tfrac{1}{#1}}

\newcommand{\sabs}[1]{| #1 |} 
\newcommand{\abs}[1]{\left| #1 \right|} 

\newcommand{\set}[1]{\{#1\}} 
\newcommand{\sets}[2]{\set{#1\,:\,#2}} 
\newcommand{\Bset}[1]{\Big\{#1\Big\}} 
\newcommand{\Bsets}[2]{\Bset{#1\,:\,#2}} 

\newcommand{\bfone}{{\mathbf 1}}
\newcommand{\ind}[1]{ \bfone_{{#1}}} 
\newcommand{\inds}[1]{ \bfone_{\set{#1}}} 
\newcommand{\seq}[1]{\set{#1_n}_{n\in\N}} 
\newcommand{\seqm}[1]{\set{#1_m}_{m\in\N}} 

\newcommand{\fml}[2]{\set{#1}_{#2}}
\newcommand{\sqm}[1]{\{ #1 \}_{m\in\N}}  
\newcommand{\norm}[1]{{||#1||}} 

\newcommand{\prfi}[1]{ \{ #1 \}_{t\in [0,\infty)}}

\newcommand{\tto}{\twoheadrightarrow}

\newcommand{\downto}{\searrow}

\newcommand{\Implies}{\Rightarrow}

\providecommand{\R}{} \renewcommand{\R}{{\mathbb R}}

\providecommand{\N}{} \renewcommand{\N}{{\mathbb N}}

\newcommand{\PP}{{\mathbb P}}

\newcommand{\EE}{{\mathbb E}}
\newcommand{\FF}{{\mathcal F}}

\newcommand{\tsum}{\textstyle\sum}

\newcommand{\FFF}{{\mathbb F}}

\newcommand{\eps}{\varepsilon}
\newcommand{\ld}{\lambda}

\newcommand{\vp}{\varphi}

\newcommand{\esl}{{\mathcal L}} 

\newcommand{\slone}{\esl^1}

\newcommand{\el}{{\mathbb L}} 
\newcommand{\lzer}{\el^0}
\newcommand{\lone}{\el^1}

\newcommand{\lpee}{\el^p}

\newcommand{\define}[1]{{\textbf{#1}}}

\newcounter{notecounter}

\newcommand{\efor}{\text{ for }}
\newcommand{\eforall}{\text{ for all }}

\newcommand{\eand}{\text{ and }}

\newcommand{\ewhere}{\text{ where }}
\newcommand{\ewith}{\text{ with }}

\newcommand{\cd}{c\` adl\` ag } 
\newcommand{\cg}{c\` agl\` ad } 

\newcommand\tf{{\tilde{f}}}

\newcommand\sC{{\mathcal C}}

\newcommand\sD{{\mathcal D}}

\newcommand\sF{{\mathcal F}}

\newcommand\sG{{\mathcal G}}

\newcommand\sL{{\mathcal L}}

\newcommand\sM{{\mathcal M}}

\newcommand\sN{{\mathcal N}}

\newcommand\sO{{\mathcal O}}

\newcommand\sP{{\mathcal P}}
\newcommand\bsP{{\boldsymbol{P}}}

\newcommand\sS{{\mathcal S}}

\newcommand\tT{{\tilde{T}}}

\newcommand\sV{{\mathcal V}}

\newcommand\sX{{\mathcal X}}

\newcommand{\Coord}{\mathsf{Coord}}
\newcommand{\QCoord}{\mathsf{QCoord}}

\newcommand{\Time}{\mathsf{Time}}
\newcommand{\Qtime}{\mathsf{QTime}}
\newcommand{\Stop}{\mathsf{Stop}}
\newcommand{\Qstop}{\mathsf{QStop}}

\newcommand{\dr}{D_{\R}}

\newcommand{\prob}{\mathsf{Prob}}
\newcommand{\kernel}{\mathsf{Kern}}
\newcommand{\Borel}{\mathsf{Borel}}
\newcommand{\univ}{\mathsf{Univ}}

\newcommand{\prst}[1]{\set{#1}_{t\in \Time}}

\newcommand{\mha}{\sM^{\#}}
\newcommand{\mfa}{\sM^{f}}
\newcommand{\mpa}{\sM^{p}}
\newcommand{\msa}{\sM^{*}}
\newcommand{\Trunc}[1]{T_{#1}}
\newcommand{\Tt}{\Trunc{t}} \newcommand{\Ts}{\Trunc{s}}
\newcommand{\Ttau}{\Trunc{\tau}} \newcommand{\Tkappa}{\Trunc{\kappa}}

\newcommand{\filt}{(\Omega,\FF,\FFF=\prst{\sF_t})}
\newcommand{\filta}{(\Omega^{\alpha},\FF^{\alpha},\FFF^{\alpha}=\prst{\sF^{\alpha}_t})}
\newcommand{\filtx}{(\Omega^{\xi},\FF^{\xi},\FFF^{\xi}=\prst{\sF^{\xi}_t})}
\newcommand{\filtax}{(\Omega^{\alpha\xi},\FF^{\alpha\xi},\FFF^{\alpha\xi}=\prst{\sF^{\alpha\xi}_t})}

\newcommand{\tOmega}{\tilde{\Omega}}
\newcommand{\hOmega}{\hat{\Omega}}
\newcommand{\tfilt}{(\tOmega,\tilde{\FF},\tilde{\FFF}=\prst{\tilde{\sF}_t})}
\newcommand{\tomega}{\tilde{\omega}}
\newcommand{\ak}{\ast_{\kappa}}
\newcommand{\at}{\ast_{\tau}}

\newcommand{\oto}{\omega\ast_t \omega'}

\newcommand{\slzm}{\sL^{0-1}}
\newcommand{\slmz}{\sL^{1-0}}
\renewcommand{\bsP}{\bar{\sP}}
\newcommand{\mlt}{\mu_{\leq \tau}}
\newcommand{\nlt}{\nu^{\leq \tau}}
\newcommand{\nelt}{(\nu^{\eps})^{\leq \tau}}

\newcommand{\ogk}{\omega_{\geq\kappa}}
\newcommand{\oemga}{\omega}
\newcommand{\cep}{C_{E^{\partial \sO}}}

\newcommand{\Oax}{\Omega^{\alpha\xi}}
\newcommand{\sGa}{\sG^{a}}
\newcommand{\Gup}{G^{\uparrow}}
\newcommand{\Cr}{C_{\R}}
\newcommand{\Grup}{\Gup_{\R}}
\newcommand{\osP}{\overline{\sP}}

\def\mathclap#1{\text{\hbox to 0pt{\hss$\mathsurround=1pt#1$\hss}}}

\title%
[A Framework for the DPP]
{A Framework for the Dynamic Programming Principle and Martingale-generated Control Correspondences}

\author{Roman Fayvisovich}
\address{Roman Fayvisovich, Department of Mathematics\\
The University of Texas at Austin}
\email{rfayvisovich@math.utexas.edu}
\thanks{
  \emph{Acknowledgments:}
  The authors would like to thank Mihai S\^\i rbu and Kasper Larsen for
  valuable conversations and acknowledge the support by the
  National Science Foundation under Grants
  DMS-0706947 (2007 - 2010),
  DMS-09556194 (2010 - 2015),
  DMS-1107465 (2012 - 2017) and
  DMS-1516165 (2015 - 2018).
  Any opinions, findings and conclusions or recommendations
  expressed in this material are those of the author(s) and do not
  necessarily reflect the views of the National Science Foundation (NSF)}
\author{Gordan \v Zitkovi\' c}
\address{Gordan \v Zitkovi\' c, Department of Mathematics\\
The University of Texas at Austin}
\email{gordanz@math.utexas.edu}

\begin{document}


\date{\today}

\begin{abstract}
We construct an abstract framework in which the dynamic programming
principle (DPP) can be readily proven. It encompasses a broad range of
common stochastic control problems in the weak formulation, and deals with
problems in the ``martingale formulation'' with particular ease. We give two
illustrations; first, we establish the DPP for general controlled
diffusions and show that their value functions are viscosity solutions of
the associated Hamilton-Jacobi-Bellman equations under minimal conditions.
After that, we show how to treat singular control on the example of the
classical monotone-follower problem.
\end{abstract}

\maketitle

\section{Introduction} The goal of this paper is creating a probabilistic
framework in which the dynamic programming principle (DPP) can be easily
proved. To be useful, such a framework needs to be sufficiently powerful,
so as to encompass as many stochastic control problems as possible, but
also sufficiently simple, so that it is easily applied in a given
situation.  On a deeper level, our intention is to identify the fundamental
properties stochastic control problems and their setups need to have in
order for the DPP to hold. One of the many interesting things about
(proving) the DPP is that its validity depends both on topological/measure
theoretic properties of the underlying spaces (such as the Polish
structure) and structural properties of the control problem (such as the
ability to concatenate controls).  A large part of this paper is a study
of their interplay in the setting of filtered probability
spaces and general formulations of stochastic-control problems.

Even though the dynamic programming principle has been introduced in the mid
20th century, or even earlier,  (we point the reader to \cite{Zit14} for a
short historical overview), research related to DPP - especially in
continuous time - underwent somewhat of a renaissance in the past several decades
(see., e.g., \cite{Elk81}, \cite{Bor89}, \cite{FleSon93}, \cite{SonTou02},
\cite{SonTou02a}, \cite{BouVu10}, \cite{BouTou11} \cite{BouNut12},
\cite{ElKTan13}, \cite{ElKTan13a} and \cite{Zit14}).

\subsection{Our contributions} Our starting point is the paper \cite{Zit14}
which focuses on two specific control problems and shows that they both
satisfy the DPP.  Therein, the so-called controlled Markov families (families
of \emph{sets} of probability measures indexed by the elements of a state
space) are introduced, and DPP is formulated as a natural analogue of the
Markov property in that setting. That formulation helps identify three
separate properties (already present in the literature, see, e.g.,
\cite{ElKTan13a,NutHan13,Zit14}) of a controlled Markov family, called
\emph{analyticity,
concatenability} and \emph{disintegrability,} under which the DPP holds.  On their
own, these three properties do not amount to much more than a rephrasing of the
DPP without making it much easier to establish. The present paper takes up
the task of providing wide sufficient conditions for each of these three
and, thus, for the validity of the DPP.

\subsubsection{Truncation- and truncation-concatenation spaces} We begin by
introducing the structure of a \emph{truncated space} (T-space) which
carries the structure of a ``measurably-filtered space'' with each $\sF_t$
generated by a single, albeit, Polish-valued, random variable.  Perhaps
unexpectedly at first, virtually all (uncompleted) concrete filtrations
used in probability and stochastic control turn out to be T-spaces;
moreover, we show that some perks of canonical spaces $C$ and $D$ (such as
Galmarino's test) extend to all T-spaces.  Another added benefit is that
sigma-algebras $\sF_{\tau}$ corresponding to stopping times inherit the
property of being generated by a single, Polish-valued random variable.
This observation simplifies many of our proofs and provides further insight into the
structure of T-spaces.  Moreover, many natural constructions (such as products
or subspaces) work well in the T-space context.  This is particular
important for our purposes as control problems come in a variety of forms,
but are invariably built out of a smaller number of ``probabilistic building blocks''.
In the same, categorical, worldview, a natural and useful notion of a
morphism between T-spaces can be introduced.

If one adds a time-indexed family of binary operations to a T-space and
imposes appropriate measurability and compatibility requirements, one
obtains the structure of a \emph{truncation-concatenation space}
(TC-space). The idea is to abstract away the main properties that define
the operation of concatenation in the context of the DPP. In addition to
the model case of pasting of (right-) continuous paths, many other forms of
concatenation are covered by TC-spaces. Indeed, while the state spaces of
control problems typically involve the spaces of (right-, left-, \dots)
continuous trajectories, the spaces of controls are much less regular and
need a more flexible framework. Just like in the case of T-spaces, one
defines products, subspaces and structure-preserving maps (morphisms)
between TC-spaces. Morphisms into the model space $D_{\R}$ of \cd{}
trajectories play an especially important role later when we deal with
martingale-generated controlled Markov families.

Once TC-spaces are set up, control problems are represented by
\emph{control correspondences}, i.e., correspondences that map each element
of the sample space into a set of probability measures on it.  In this
context, one defines the notion of a value function of a control problem,
as well as the properties of analyticity, concatenability and
disintegrability which, together, imply (an abstract) DPP. It is, perhaps,
interesting to note that no notion of a state is needed for the abstract
DPP to hold. It can be introduced explicitly, as we often do, but its role
is abstractly taken over by the notion of compatibility used to define a TC
space.

\subsubsection{Martingale-generated control correspondences.} Our
central claim is that truncation-con\-ca\-te\-na\-ti\-on spaces, together with
a shift operator (which can be thought of as a partial inverse of
concatenation and plays a central role in the study of disintegrability),
provide a convenient framework on which a variety of stochastic control
problems can be posed and analyzed. Of course, the validity of the DPP will
depend on the nature of the problem itself, but, as we show in examples,
this amounts to a verification of a small number of easily checked
intuitive conditions. Focusing mainly on control problems in their weak
formulation, and the derived control correspondences, we identify two
important cases in which these conditions are especially easy to check. One
is when the probability of the future evolution is controlled
directly, without the need for an intermediate ``control process'', as is the
case, e.g.,  with pure singular-control problems.  In the other, much
larger, family of cases, explicit control processes are typically present,
but their structure is such that access to the totality of all
possible controlled dynamics is possible via a system of ``well-behaved''
constraints. Such constraints are often expressible in terms of the (local)
martingale property of a class of real-valued \cd{}processes. The control
correspondences constructed in this way are said to be
\emph{martingale-generated} as
they correspond, loosely, to what is known as the martingale formulation of
optimal control in the literature. The second third of the paper focuses on
martingale-generated control correspondences on TC spaces and provides
sufficient conditions on the structure of the constrains (by interpreting
them as morphisms into the model space $D_{\R}$) for the DPP to
hold.

\subsubsection{Examples} The final third of the paper contains two examples
meant to illustrate the versatility of our framework. The first one is the
classical controlled-diffusion case which we consider in the weak
formulation and place it in our setting as a martingale-generated control
correspondence. We show that sufficient conditions established in the
previous section apply in this case, and conclude that the DPP holds under
minimal conditions on the coefficients and the form of the controls.  We also
demonstrate that value functions of such control problems are viscosity solutions
of the corresponding Hamilton-Jacobi-Bellman equations, under slightly
stronger conditions (continuity of coefficients and admissibility of
locally constant controls). This partially
generalizes several recent
results in the literature, such as the ``stochastic Perron'' method of
Bayraktar and S\^{i}rbu (introduced in \cite{BaySir12}) or the work of
Bouchard and Touzi on the ``weak DPP'' (see \cite{BouTou11}). 
The same class of problems - under a somewhat different set of assumptions - has
already
been treated by the authors of \cite{ElKTan13,ElKTan13a}. Like the present
paper,
they rely on the ability to pose an equivalent controlled martingale problem
on a suitable canonical space and characterize the resulting control
correspondence using at most countably many test functions.

Our second example is of singular type, and features a mildly
generalized Monotone-Follower problem. Here, we not only show how to
establish the DPP for a singular-control problem in our framework, but also
showcase its flexibility. Indeed, we split the
variables into two groups and deal with one directly, and with the other
using the martingale-generated approach. These two are considered separate
control problems (with separate control correspondences) until the very
last moment when they are easily merged.

\subsection{Notation and conventions.} Both probabilistic and analytic
tools -  which often come with less-than-perfectly compatible notations and
terminology - are used in this paper. For the convenience of the reader, we
outline some of our major choices and conventions below.

Both probabilistic $\EE^{\PP}[X]$ and analytic $\int G\, d\mu$ notation for
integration will be used. The former will appear mostly in examples, and
the latter in the abstract part.

Many of our probability spaces come with \define{Polish} (completely
metrizable, separable) sample spaces and Borel probability measures. When
the Polish structure is present, measurability will always refer to the
associated Borel $\sigma$-algebra, denoted by $\Borel(\Omega)$.
The
set of all probability measures on $\Borel(\Omega)$ is denoted by
$\prob(\Omega)$.

A subset $A$ of a Polish space $\Omega$ is called \define{analytic} if it
can be realized as a projection of a Borel subset of $\Omega\times \R$ onto
$\Omega$. We remind the reader that analytic subsets of Polish spaces are
closed under countable unions, intersections and products, but not
necessarily under complements. It will be important for us that each
analytic set is in the \define{universal} $\sigma$-algebra - denoted by
$\univ(\Omega)$ - i.e., the family of all sets which belong to the
completion $(\Borel(\Omega))^*_{\mu}$ for each $\mu\in\prob(\Omega)$.
 We
refer the reader to \cite{Sri98} for all the necessary details concerning
descriptive set theory (see also \cite{BerShr78} for a thorough treatment
of related topics in the context of the dynamic programming principle).

We topologize $\prob(\Omega)$ with the topology of
(probabilist's) weak convergence. This way, $\prob(\Omega)$ becomes a
Polish space.  The following well-known fact, proved in a standard way via
the monotone-class theorem, will be used throughout without  mention: Let
$U$ and $V$ be Polish spaces and let $f:U\times V \to [0,\infty]$ be a
Borel-measurable function. The map
\[ U \times \prob(V) \ni (x,\mu) \mapsto
\EE^{\mu}[f(x,\cdot)]=\int_{V} f(x,y)\, \mu(dy)\]
is Borel measurable.

A probability measure defined on $\Borel(\Omega)$ admits a natural
extension to $\univ(\Omega)$. Similarly, our kernels \emph{will always be
universally measurable}. More precisely, for Polish spaces
$\Omega,\tilde{\Omega}$, a map $\nu:\Omega\times \Borel(\tilde{\Omega}) \to
[0,1]$ is called a \define{kernel} if $\nu(\omega,\cdot) \in
\prob(\tilde{\Omega})$ for each $\omega\in \Omega$ and $\nu(\cdot,B)$ is a
universally-measurable map on $\Omega$, for each
$B\in\Borel(\tilde{\Omega})$.  Depending on the situation we use both
notations $\nu(\omega,\cdot)$ and $\nu_{\omega}$ for the probability
measure associated by $\nu$ to $\omega$.

A \define{standard Borel space} is, by definition, a measurable space which
admits a measurable bijection to a Borel subset of some $\R^n$, whose
inverse is also measurable (a \define{bimeasurable isomorphism}). All
standard Borel spaces of the same cardinality are bimeasurably isomorphic,
and so, each standard Borel space can be given a complete and separable
(Polish) metric so that the induced measurable structure matches the
original one. With this in mind, we talk of standard Borel spaces when only
the measurable structure is relevant, and about Polish spaces when
topological properties are required.

\section{An abstract setting for the Dynamic Programming Principle (DPP)}
\label{sec:abs}
Let the \define{time set} $\Time$ be either $[0,\infty)$ or $\N_0$.  An
overwhelming majority of applications will only use these two time sets, so
we do not aim for greater generality. We do note that the results of this
section will hold for more general time structures (such as intersections
with $[0,\infty)$ of Borel-measurable additive subgroups of $\R$).

\subsection{T-spaces (truncated spaces)} We start with the definition of
T-spaces - a class of filtered probability spaces our analysis will be
based on.
\begin{definition}[T-spaces]
\label{def:trunc}
A filtered measurable space $\filt$ is called a \define{T-space}
(or a \define{truncated space}) if
\begin{enumerate}
  \item $(\Omega,\sF)$ is a standard Borel space and
  $\FF =\bigvee_{t\in\Time} \sF_t$.
  \item there exists a family
  $\prst{\Tt}$ of maps $T_t:\Omega\to\Omega$ -  called a \define{truncation}
  - such that
\begin{enumerate}
 \item \label{ite:meas-trunc}
  $(t,\omega) \mapsto \Tt(\omega)$ is (jointly) measurable,
 \item
  \label{ite:proj-trunc}
  $\Tt \circ \Ts = \Trunc{s\wedge t} \eforall s,t\in\Time$, and
\item $\sF_t = \sigma(T_t)$ for each $t\in \Time$.
 \end{enumerate}
 \end{enumerate}
\end{definition}
For notational reasons, we always add the identity map $T_{\infty}=\Id$
to any truncation. Moreover, we often use the
alternative notation $\omega_{\leq t}$ for $T_t(\omega)$.

\subsection{First examples of T-spaces}\
\label{sss:exam-T}
All T-spaces are necessarily countably generated, so not every filtered
probability space can be endowed with the structure of a T-space.
Nevertheless, as our examples below aim to show, many spaces used in
stochastic analysis and optimal stochastic control are natural T-spaces.
When it is necessary to make a distinction, we take $\Time=[0,\infty)$ and
leave it to the reader to make the necessary minor adjustments needed for
the case $\Time=\N_0$. Once we describe various natural constructions
involving T-spaces in subsection \ref{sss:categ} below, the reader will be
able to produce many more examples.

\subsubsection{The path space $D_E$.} Let $E$ be a Polish space, and let
$D_E$ denote the family of all \cd functions from $\Time$ to $E$.  For
$t\in\Time$, we define the truncation map $T_t:D_E \to D_E$ by
 \begin{align}
 \label{equ:standard-trunc}
   T_t(\omega) (s) =\omega(t\wedge s) \efor s\in\Time,
 \end{align}
so that \eqref{ite:proj-trunc} of Definition \ref{def:trunc} holds.  It is
well-known that $D_E$ is a Polish space under the Skorokhod topology.  The
map $T_t$ is  Skorokhod-continuous, and therefore, measurable.  Hence, as a
Caratheodory function, $T:\Time \times \Omega \to \Omega$ is (jointly)
measurable. The filtration $\sF_t=\sigma(T_t), t\in\Time$ clearly coincides
with the (raw) filtration generated by the coordinate maps $\omega \mapsto
\omega(t)$.

\subsubsection{Path spaces $G_E$, $C_E$ and $\Lip^{L,x_0}_{\R}$.}
\label{par:G}
Analogous constructions can be performed on the space $G_E$ of
left-con\-ti\-nu\-ous and right limited paths from $\Time$ to $E$, or on
the space $C_E$ of continuous paths.  Both of these are given the Skorokhod
topology (and the induced Borel structure), which, in the case of $C_E$
reverts to the usual topology of locally uniform convergence. Unless
specified otherwise, these spaces (and their subspaces) will always be
endowed with the standard truncation given by \eqref{equ:standard-trunc}.

We will also have use for
the  space $\Lip^{L,x_0}_{\R}$ consisting of all functions $x:[0,\infty)
\to \R$ such that $x(0)=x_0$ and $\abs{x(t) - x(s)}\leq L\abs{t-s}$ for all
$s,t\in [0,\infty)$. It is easy to see that $\Lip^{L,x_0}_{\R}$ is also a
T-space with the standard truncation.

\subsubsection{The space $\lzer_A$ and related spaces.} Let $A$ be a standard Borel space, let
$\ld$ be the Lebesgue measure (or any other Radon
measure) on $[0,\infty)$, and let $\hat{\ld}$ denote an equivalent
probability measure on $[0,\infty)$ (e.g., $\hat{\ld}(dt)=e^{-t}\,
\ld(dt)$, when $\ld$ is the Lebesgue measure).
We define $\lzer_A$ as the set of all
$\ld$-a.e.-equivalence classes of Borel functions $\alpha:[0,\infty)\to A$.
Given a bimeasurable
isomorphism $\phi: A \to [-1,1]$ (which exists thanks to
the standard Borel property of $A$) we metrize $\lzer_A$ by
\[ d(\alpha,\beta) = \norm{\phi(\alpha) -
\phi(\beta)}_{\lone(\hat{\ld})}.\]
This way, $\Omega=\lzer_A$ becomes a Polish
space and a natural truncation on it is defined by
  \[ T_t(\alpha) = \begin{cases}\alpha_u, & u<t\\ \phi^{-1}(0), & u\geq t.
  \end{cases}\]
We note that the equivalence class of the right-hand side depends on
$\alpha$ only through its equivalence class, and that, while $d$ and the
induced Polish topology depend on
the choice of $\phi$ and $\hat{\ld}$, the resulting standard
Borel structure does not. The choice of this particular $\phi$ makes it easy
to show that $T_t$ is jointly measurable; indeed, it will be
continuous under $d$ in both of its arguments.

Once the space $\lzer_A$ is constructed, one can easily show that many of
subsets (such as the $\lpee$ spaces when $A=\R$) are also T-spaces.

\subsubsection{Spaces of measures.}
For a metrized Polish space $U$, let $\mha(U)$
be the family of all {boundedly-finite Borel measures}
on $U$, i.e,. those measures
$\mu$ such that $\mu(B)<\infty$, as soon as $B$ is a bounded Borel set.
There exists a metric on $\mha(U)$, whose topology coincides with the
topology of weak convergence when restricted on measures supported by a
fixed bounded set (see \cite[Section A2.6, p.~402]{DalVer03} for the proof
of this and other statements about the space $\mha(U)$ we make below).
Under the full topology induced by this metric, called the
$w^{\#}$-topology, $\mha(U)$ becomes a Polish space.
Moreover, a sequence $\seq{\mu}$ in $\mha(U)$
converges if and only if $\int f\, d\mu_n \to \int f \mu$ for each bounded
and continuous function $f:\Omega\to\R$ which vanishes outside a bounded set.
The Borel $\sigma$-algebra on $\mha(U)$ is generated by the evaluation maps
$\mu \mapsto \mu(A)$, where $A$ ranges over a family of all bounded Borel
subsets of $U$. The subsets $\mfa(U)$ and
$\mpa(U)=\prob(U)$ of
$\mha(U)$, consisting only of finite or probability measures
(respectively),  are easily seen to be Borel subsets of $\mha(\Omega)$,
and, therefore, standard Borel spaces
themselves.

For a Polish space $E$, we set $\Omega=\msa(U)$, where $U=[0,\infty) \times
E$ and $* \in \set{\#, f, p}$. The truncation maps are given by
\[ \mu_{\leq t} (A) = \mu\Big(  ([0,t]\times E) \cap A\Big), \efor t\in
[0,\infty), A\in [0,\infty) \times E.\]
With the filtration generated by the maps $T_t$, it is clear that $\vee_t
\sF_t$ is the Borel $\sigma$-algebra on $\Omega$. The only remaining property from Definition \ref{def:trunc}
is \eqref{ite:meas-trunc}, for which it is
sufficient to note that for any boundedly supported function $f$ we have
$\int f\, d\mu_{\leq t} = \int f \ind{[0,t]\times E}\, d\mu$. Indeed, it
follows that $(t,\mu)\mapsto \mu_{\leq t}$ is a Caratheodory function as
it is right continuous in $t$ and measurable in $\mu$.

\subsubsection{Predictable truncations.} In many the examples above, it is possible
to define several different truncations on the same underlying Polish
space. For example, in the case of the canonical space $D_E$, we may set
\[ T'_t(\omega)(s) = \begin{cases} \omega_s,& s<t \\ \omega_{t-}, & s\geq t
\end{cases}.\]
It is easily checked that $T'_t$ is indeed, a truncation on $D_E$; we call
it the \define{predictable truncation}.

\subsection{Truncating at stopping times}\
Given a T-space $\filt$,  let the set of all stopping times be denoted by $\Stop$. The index set for the family of truncation operators
can be extended to $\Stop$ by setting
\begin{align*}
T_{\tau} (\omega) = T_{\tau(\omega)}(\omega) \efor \tau\in\Stop
\eand
  \omega\in\Omega,
\end{align*}
where the convention that $T_{\infty}$
is the identity map is used.
As is the case with deterministic times, the notation $T_{\tau}(\omega)$ will often be replaced
by the less cumbersome (and more suggestive) $\omega_{\leq \tau}$.
\begin{proposition}
\label{pro:prop-trunc}
For all $t\in\Time$, $\omega\in\Omega$, $\tau,\kappa \in\Stop$
and  we have
\begin{enumerate}
\item \label{ite:comp-stop}
$T_{\tau}$ and $T_{\kappa}$ are measurable maps on $\Omega$ and
$\Ttau \circ \Tkappa = \Trunc{\tau \wedge \kappa}$
\item \label{ite:fix-filtr}
$\sigma(T_\tau) = \sets{A\in \sF}{ T_\tau^{-1}(A) = A}$,
and
\begin{align*}
 \text{ ``$A\in\sigma(T_\tau)$'' \quad is equivalent to \quad }
 ``\omega\in A\  \Leftrightarrow \ \omega_{\leq\tau}\in A''
 \end{align*}
\item \label{ite:tau-omega}
$\tau(\omega) = \tau(T_{\tau}(\omega))$,
and hence $\tau$ is $\sigma(T_\tau)$-measurable
\item \label{ite:tau-sigma}
$\sigma(T_{\tau}) = \sF_{\tau}$,
where $\sF_\tau=\sets{A\in \sF}{A\cap \set{\tau\leq
t}\in \sF_t, \forall t\in\Time}$
\item \label{ite:F_t-meas}
Let $(S,\sS)$ be a standard Borel space.
An $(\sF,\sS)$-measurable map $Z:\Omega\to S$
is $(\sF_{\tau},\sS)$-measurable if and
only if $Z\circ T_{\tau} = Z$.
\end{enumerate}
\end{proposition}
\begin{proof}\
\begin{enumerate}
\item
Measurability of $\Ttau$ follows directly from the measurability
of stopping times and
the joint measurability of $(t,\omega)\mapsto
\Tt(\omega)$ on $(\Time\cup\set{\infty}) \times \Omega$.
Applying Definition \ref{def:trunc}, part \eqref{ite:proj-trunc}
pointwise for $t=\tau(\omega)$ and $s=\kappa(\omega)$
gives $\Ttau \circ \Tkappa = \Trunc{\tau \wedge \kappa}$.
\item
By part (1) we have $T_\tau = T_\tau\circ T_\tau$ for each $\tau\in\Stop$,
and so for any $A\in\sF$, we have
\begin{align*}
	A = T_\tau^{-1}(B) \text{ for some $B\in \sF$}
	\quad\Leftrightarrow\quad
	A = T_\tau^{-1}(A).
\end{align*}
Furthermore the condition $A=T_\tau^{-1}(A)$ is equivalent to:
$$ \omega\in A \quad\Leftrightarrow\quad \omega_{\leq\tau}\in A $$
\item
Fix $\omega\in\Omega$, let $t=\tau(\omega)$, and let $A=\{\tau=t\}$.
Since $\tau\in\Stop$, then $A\in\sF_t=\sigma(T_t)$.
Combining part (2) with the fact that $\omega\in A$
implies $T_t(\omega)\in A$. Therefore:
$$ \tau(T_{\tau(\omega)}(\omega)) =
  \tau(T_t(\omega)) = t = \tau(\omega) $$
\item
For the forward inclusion, let $A\in\sigma(T_\tau)$.
Thanks to \eqref{ite:fix-filtr} above, we have $A=T_{\tau}^{-1}(A)$.
Therefore for all $t\in\Time$ we have:
 \begin{align*}
   A\cap \set{\tau\leq t}
   &= \sets{ \omega\in\Omega}{ T_{\tau(\omega)}(\omega) \in A,
   \tau(\omega)\leq t} \\
   &= \sets{ \omega\in\Omega}{ T_{\tau(\omega)\wedge t}(\omega) \in A,
   \tau(\omega)\leq t} = T_{\tau\wedge t}^{-1}(A) \cap \set{\tau\leq
   t}\in\sF_t,
 \end{align*}
where we used the fact that $T_{\tau\wedge t}  = T_{\tau \wedge t} \circ
T_{t}$ is $\sF_t$-measurable.
Therefore $A\in\sF_\tau$,
and hence $\sigma(T_\tau)\subset\sF_\tau$.
\vspace{2mm}\\
For the backward inclusion, let $A\in\sF_\tau$.
By part \eqref{ite:fix-filtr}, it suffices to show:
$$ \omega\in A \quad\Leftrightarrow\quad \omega_{\leq\tau}\in A $$
First suppose $\omega\in A$ and let $t=\tau(\omega)$.
Since $A\in\sF_\tau$,
then $\omega \in A \cap \{\tau \leq t\} \in \sF_t$.
Applying \eqref{ite:fix-filtr} to $A \cap \{\tau \leq t\}$ gives
$ \omega_{\leq\tau} \in A \cap \{\tau \leq t\} \subset A $.
\vspace{2mm}\\
For the other direction, suppose $\omega_{\leq\tau}\in A$.
By part \eqref{ite:tau-omega} we have $\tau(\omega_{\leq\tau})=\tau(\omega)$
and hence $\omega_{\leq\tau} \in A\cap\{\tau\leq t\} \in \sF_t$.
Applying \eqref{ite:fix-filtr} to $A \cap \{\tau \leq t\}$ gives
$ \omega \in A \cap \{\tau \leq t\} \subset A $.
\item If $Z=Z\circ T_{\tau}$, then $Z$ is $\sF_{\tau}$-measurable as a measurable
transformation of the $\sF_{\tau}$-measurable map $T_{\tau}$.
Conversely, if $Z$ is
$\sF_{\tau}$-measurable, the standard Borel property and the Doob-Dynkin
lemma guarantee the existence of a measurable map $\zeta: \Omega\to S$ such
that $Z = \zeta\circ T_{\tau}$. A composition with $T_{\tau}$ yields
that
\[ Z\circ T_{\tau} = \zeta\circ T_{\tau} \circ T_{\tau} = \zeta \circ
T_{\tau} = Z.\qedhere\]
\end{enumerate}
\end{proof}

\subsection{Constructions on T-spaces}\
\label{sss:categ}

Next, we describe several natural notions and
constructions on T-spaces, as well as various operations that produce new
T-spaces from the old ones. For the remainder of this subsection, let
$\filt$ and $\tfilt$ be two T-spaces, with truncations $\prst{T_t}$ and
$\prst{\tT_t}$, respectively.

\subsubsection{Structure-preserving maps}
A useful structure-preserving notion in the case of T-spaces turns out to
be non-anticipation:
\begin{definition}
A measurable map $F:\Omega \to\tilde{\Omega}$ is said to be
\define{non-anticipating}
if it is $(\sF_t,\tilde{\sF}_t)$-measurable, i.e.
  $F^{-1}(\tilde{\sF}_t) \subseteq \sF_t$ for each
  $t\in\Time$.
  \end{definition}
We have the following characterization using the truncation maps:
\begin{proposition}
\label{pro:T-morph-char}
A measurable map $F:(\Omega,\sF) \to (\tOmega,\tilde{\sF})$ is
non-anticipating if and only if
\[\tT_t \circ F \circ T_t = \tT_t\circ F \eforall t\in\Time.\]
\end{proposition}
\begin{proof}
By Proposition \ref{pro:prop-trunc} part \eqref{ite:fix-filtr} we have
$\tilde\sF_t = \sigma(\tilde T_t) = \tilde T_t^{-1}(\tilde\sF)$,
and by part \eqref{ite:F_t-meas} we have
$\tilde T_t\circ F$ is $\sF_t$-measurable if and only if
$\tT_t \circ F \circ T_t = \tT_t\circ F$.
Therefore for all $t\in\Time$:
\begin{align*}
  F^{-1}(\tilde\sF_t) \subset \sF_t
  &\quad\Leftrightarrow\quad
  F^{-1} (\tilde T_t^{-1} (\tilde\sF)) \subset \sF_t \\
  &\quad\Leftrightarrow\quad
  \tilde T_t \circ F \text{ is $\sF_t$-measurable} \\
  &\quad\Leftrightarrow\quad
  \tT_t \circ F \circ T_t = \tT_t\circ F \qedhere
\end{align*}
\end{proof}
\begin{remark}
\label{rem:T-morph}
One could also consider an alternative notion of a structure-preserving map
where we
require that $\tilde{T}_t\circ F = F \circ T_t$ for all $t\in\Time$. Proposition
\ref{pro:T-morph-char} and the fact that $\tilde{T}_t \circ \tilde{T}_t =
\tilde{T}_t$ imply that T-morphisms are non-anticipating, but the converse
is not true.

\end{remark}
\subsubsection{T-subspaces.}

We say that a T-space $\tfilt$ is a \define{$T$-subspace} of $\filt$ if
$\tOmega \subseteq \Omega$ and $\tilde{\sF}_t \subseteq \sF_t$, for all
$t\in \Time$. As the following result show, subsets preserved by
truncation inherit a structure of a T-space:
\begin{proposition}
\label{pro:sub}
Let $\filt$ be a T-space, and let $\Omega'$ be a measurable subset of
$\Omega$ with the property that $T_t(\Omega') \subseteq \Omega'$, for all
$t\in\Time$. Then the family $\prst{T'_t}$ given by $T'_t = T_t|_{\Omega'}$,
is a truncation,
and the filtered space $(\Omega', \sF', \prst{\sF'_t})$,
given by $\sF' = \sets{B\in \sF}{B\subseteq \Omega'}$, $\sF_t =
\sigma(T'_t)$, $t\in \Time$, is a T-space and a subspace of $\filt$.
\end{proposition}
\begin{proof}
Clearly $(\Omega', \sF')$ is a subspace of $(\Omega, \sF)$.
To satisfy Definition \ref{def:TC} of T-spaces,
note that part (1) follows from the construction of $\Omega'$ and $\sF'$,
and the properties of part (2) are passed down from $T$ to $T'$.
\end{proof}
\begin{example}
\label{exa:subspaces}
Truncation operators on $D_E$ leave invariant
several important measurable subsets of $D_E$. Among the examples are
\begin{enumerate}
\item $C_E$, the family of all everywhere continuous elements of $D_E$,
\item $D^{E_0}_{E}$, the family of paths in $D_E$ which start from a point
in $E_0$, and
\item $D_{E^F}$, the family of paths in $D_E$ stopped once they hit the
closed subset $F$ of $E$.
\item $D^{fv}\, (D^{\uparrow}, D^{\downarrow})$, the family of all paths in $\dr$
all of whose components are of finite variation (nondecreasing,
nonincreasing)
\item $\Lip^{L}_{\R}$, the family of all Lipschitz continuous maps from
$[0,\infty)$ to $\R$, with the Lipschitz constant at most $L$.
\end{enumerate}
More examples can be produced by various intersections of the above sets.
\end{example}
\subsubsection{Products.}
\label{par:prod}
T-spaces behave well under products, too. Indeed,
the standard Borel space $\hOmega =
\Omega \times \tOmega$ admits a natural truncation given by
the family
$\prst{\hat{\Tt}}$ of maps on $\hOmega$ defined by
\begin{equation}
  \label{equ:prod-maps}
  \begin{split}
    \tilde{\Tt}(\omega,\tomega) = (\Tt(\omega), \tilde{T}_t(\tomega)).
  \end{split}
\end{equation}
The resulting T-space $\hOmega$, together with the natural filtration
generated by $\prst{\hat{T}_t}$,  is called
the \define{product} of the
truncated spaces $\Omega$ and $\tOmega$. It is not difficult to see that
the same construction can be applied to countable products of truncated
spaces.

\subsubsection{State maps.} \label{sss:state}
A measurable map $X:\Omega\to E$, where $E$ is a Polish space
is called a \define{state map}. Such maps define a class of
progressively measurable $E$-valued stochastic
processes on $\Omega$ via
\[ X_t(\omega) = X(T_t(\omega)), t\in \Time\cup\{\infty\}, \omega\in\Omega\]
(where the convention $T_\infty(\omega)=\omega$ is used).
We can also write $X_{\tau}$ for $X\circ T_{\tau}$ when $\tau\in \Stop$.
\begin{remark}
\label{rem:state-map}
Our notion of a state corresponds intuitively to that used in the theory of
Markov processes, even though we insist upon assigning a state to each
$\omega\in\Omega$. If one pictures $T_t(\omega)$ as trajectory $\omega$
stopped at $t$, then $X_t(\omega)$ is simply the ``state'' at which
$\omega$ is stopped. When $\omega$ is not necessarily in the image of some
$T_t$, we assign the state abstractly imagining it to be the ``value of
$\omega(\infty)$''.
\end{remark}
\subsubsection{Actions on measures and kernels}
For a probability measure $\mu\in \prob(\Omega)$, and a stopping time
$\tau\in \Stop$ we define the
\define{truncated measure} $\mlt$ as the push-forward of $\mu$ via the truncation map $T_{\tau}$.

Two analogous operations can be applied to kernels $\nu$ from $\Omega$ to
$\Omega$. We can truncate the second argument, leading to the
\define{truncated kernel} $\nu_{\leq \tau}$, where, for each
$\omega\in\Omega$, $\nu_{\leq \tau}(\omega,\cdot)$ is the truncation of the
measure $\nu(\omega,\cdot)$, as above. On the other hand, we can define
the \define{restricted kernel} $\nu^{\leq
\tau}$ by truncating in the first argument, i.e., by setting
\begin{align*}
\nlt(\omega, B) = \nu( T_{\tau}(\omega), B).
\end{align*}
That $\nlt$ is, indeed, a kernel follows from the
fact that a Borel measurable function (like $T_{\tau})$
between two Polish spaces remains measurable under the pair of universal
$\sigma$-algebras (see \cite[Proposition 7.44, p.~172]{BerShr78}).

\subsection{TC-spaces (truncation-concatenation spaces)}
\begin{definition}
\label{def:TC}
A \define{truncation-concatenation space} (or a \define{TC-space}) is a
truncation space $\filt$
together with
a
measurable subset $\sC \subseteq
\Omega \times \Time \times \Omega$ - called the \define{compatibility set}
- and a
measurable map $\ast: \sC \to \Omega$ -
called the \define{concatenation operator},
such that the following
conditions hold:
\begin{enumerate}
\item for all $\omega,\omega' \in \Omega$ and $s,t\in\Time$ we have
  \begin{align}
  \label{equ:TC-comp}
    (\omega,t,\omega')\in\sC
    \Leftrightarrow
    (\omega_{\leq t},t,\omega') \in \sC
    \Leftrightarrow
    (\omega, t, \omega'_{\leq s})\in \sC.
  \end{align}
\item if $(\omega,t,\omega')\in\sC$, then, for all $s\in \Time$ we have
\begin{align}
\label{equ:TC-1}
   \omega\ast_t \omega' &= \omega_{\leq t} \ast_t \omega', \text{ as well
   as } \\
\label{equ:TC-2}
(\omega\ast_t \omega')_{\leq s} &= \begin{cases}
\omega_{\leq s}, & s\leq t \\ \omega \ast_t \omega_{\leq s-t}, & s> t
\end{cases}
\end{align}
\end{enumerate}
\end{definition}
The action of the concatenation operator on the triplet
$(\omega,t,\omega') \in \sC$ is denoted by $\omega \ast_t \omega'$ and
is usually interpreted as an element of
$\Omega$ ``obtained by following
$\omega$ until time $t$, with $\omega'$ attached afterwards''.
The set $\sC$ - the domain of $\ast$ - may encode a compatibility relation
necessary for the concatenation to be possible.
The set of all $\omega' \in \Omega$ such that $
(\omega,t,\omega') \in \sC$ is denoted by $\sC_{\omega,t}$, and we say that
\define{$\omega'$ is
compatible with $\omega$ at $t$} if
$\omega'\in\sC_{\omega,t}$.

In many examples
compatibility is established via a state map (as defined in subsection
\ref{sss:state} above):
\begin{definition}
\label{def:comp-X}
Given a TC-space $\filt$ and a state map $X$, we say that
the concatenation operator $\ast$
\begin{enumerate}
\item \define{factors through $X$} if \tabto{0.25\linewidth}
$X_t(\omega) = X_0(\omega')$ \tabto{0.45\linewidth} $\Implies (\omega,t,\omega')\in\sC$, and
\item \define{is a factor of $X$} if  \tabto{0.25\linewidth}
$(\omega,t,\omega')\in\sC$ \tabto{0.45\linewidth} $\Implies X_t(\omega) = X_0(\omega')$.
\end{enumerate}
\end{definition}

When needed, we also define $\omega \ast_{\infty} \omega'=\omega$,
declaring, implicitly, any two elements of $\Omega$
compatible at $t=\infty$, so that $\sC_{\omega,\infty}=\Omega$.  This way,
as in the case of the truncation spaces, the time-set $\Time$ can be extended
to the set of all stopping times by setting:
\begin{align}
\label{equ:conc-prop}
	\omega \at\omega' = \omega \ast_{\tau(\omega)} \omega'
        \efor \omega' \in \sC_{\omega,\tau(\omega)}.
\end{align}
By Proposition \ref{pro:prop-trunc}, part (3), $\tau(\omega_{\leq \tau}) =
\tau(\omega)$, and, so,
the stopping-time analogue of \eqref{equ:TC-1} holds in TC spaces:
\begin{align*}
		 \omega\at \omega' = \omega \ast_{\tau(\omega)} \omega'
		 = \omega_{\leq \tau(\omega)} \ast_{\tau(\omega)} \omega'
		 = \omega_{\leq \tau} \ast_{\tau(\omega_{\leq \tau})} \omega'
		 = \omega_{\leq \tau} \at \omega'.
\end{align*}

\subsection{Examples of TC-spaces}\

\label{sss:exam-TC}
We go through the list
of examples of T-spaces from subsection \ref{sss:exam-T} and describe how a
natural concatenation operator can be introduced.

\subsubsection{Strict concatenation on path spaces $D_E$ and $C_E$.}
\label{par:strict-conc}
We consider the space $D_E$ with the truncation $\omega_{\leq t}(s) =
\omega(s\wedge t)$.
The \define{strict concatenation} operation $\bullet$ is given by
 \begin{align}
 \label{equ:str-form}
    (\omega \bullet_t\omega')_s = \begin{cases}
     \omega(s), & s\leq t \\ \omega'(s-t), & s>t,
   \end{cases}
 \end{align}
for $\omega,\omega'\in D_E$,
where $\omega$ and $\omega'$ are considered $t$-compatible if and only if
$\omega(t) = \omega'(0)$. To check that $\bullet$ is, indeed, a
concatenation is straightforward, and we only remark that the joint
measurability of $\bullet$ (in all three of its arguments) follows from the
observation that, as a function of the inner argument $t$, it is
right-continuous in the Skorokhod topology.
When applied on its compatibility set $\sC$,
the operation $\bullet$ preserves continuity, so
it can be used to define a concatenation operator on $C_E$, as well.
Finally,
it is straightforward that
\[ X(\omega) := \liminf_{t\to\infty} \omega(t) \]
defines
an $E=\bar{\R}$-valued state map with the property $X_t(\omega) = \omega(t)$
for $t\in\Time$ and such that the concatenation operator $\bullet$
factors through it.

\begin{remark}
Many subspaces of $D_E$, in addition to $C_E$, are
closed under the strict concatenation. The reader will easily check that
all the spaces in Example
\ref{exa:subspaces} have this property; it follows that they are
TC-spaces
themselves.
\end{remark}

\subsubsection{Adjusted concatenation on $D_E$ and $C_E$}
When $E$ admits an additive structure, we can define
another concatenation operator on it, namely the \define{adjusted
concatenation} operator $\star$. It is given for $\omega,\omega'\in D_E$ by
 \begin{align}
 \label{equ:adj-concat}
    (\omega \star_t \omega')_s = \begin{cases}
     \omega(s), & s\leq t \\ \omega(t) + \omega'(s-t) -\omega'(0), & s>t,
   \end{cases}
 \end{align}
with no restrictions on compatibility, i.e., with $\sC = \Omega \times
\Time \times \Omega$. It is clear that the strict and the adjusted
concatenation operators agree on the compatibility set of $\bullet$, and
that $\ast$ can be restricted to $C_E$ without loosing any properties
required of a concatenation.

\subsubsection{Spaces of measures.}
We define the concatenation operator $\ast$ on the space
$\Omega=\sM^{\#}([0,\infty)\times E)$, described in subsection
\ref{sss:exam-T} as follows. For $\mu,\mu'\in \Omega$, we set
\[ (\mu \ast_t \mu')(A)=
\mu\Big( ([0,t)\times E) \cap A \Big )+
\mu'\Big( \big(([t,\infty)\times E) \cap A\big) - t \Big ),\]
where $ B-t = \sets{(x,s-t)}{(x,s)\in B}$, for $B\subseteq [t,\infty)\times
E$. No compatibility restrictions are imposed.
There should be no difficulty in checking that $\ast$ satisfies all
defining properties of a concatenation. We also note that the same
construction applies when $\sM^{\#}$ is replaced by $\sM^{f}$.

In the case
when $\sM^{p}$ is considered, the above operation does not preserve total
mass. This cannot be fixed by restricting compatibility, but can be
overcome   by
  defining another concatenation operation as follows:
\[ (\mu \, \tilde{\ast}_t\, \mu')(A)=
\mu\Big( ([0,t)\times E) \cap A \Big )+
\Big(1-\mu([0,t)\times E)\Big) \mu'\Big( \big(([t,\infty)\times E) \cap A\big) - t \Big ),\]

\subsubsection{$\lzer_A$ spaces.} When the underlying measure $\ld$ is the Lebesgue measure, we usually concatenate $\lzer_A$ functions as follows:
\[ (f\ast_t g)_u = \begin{cases} f_u, & u\leq t \\ g_{u-t}, & u > t
\end{cases},\]
with no compatibility restriction.
\subsection{Constructions and structure-preserving maps on TC spaces} Like
T-spaces, TC-spaces come with natural subspace and product constructions.
Their properties extend those of naked T-spaces in a predictable way, so we
skip any further discussion. The following notion of
a structure-preserving map on TC spaces will play a major role in
Section \ref{sec:mart-gen} below.
\begin{definition}
A measurable map $F:\Omega \to \tOmega$ between two TC-spaces, with
concatenation operators $\ast$ and $\tilde{\ast}$ (and compatibility sets
$\sC$ and $\tilde{\sC}$) is
called a \define{TC-morphism} if
\label{def:TC-morph}
\begin{enumerate}
  \item $F$ is non-anticipating, and
\item for all $t\in\Time$,  and all $\omega,\omega'\in\Omega$ with $\omega' \in
  \sC_{\omega,t}$ we have $F(\omega') \in \tilde{\sC}_{F(\omega),t}$ and
  \[ F(\omega \ast_t \omega') = F(\omega) \, \tilde{\ast}_t\,  F(\omega').\]
  \end{enumerate}
\end{definition}
\subsection{Concatenation of measures in TC-spaces}\

\label{sss:conc-meas}
The ability to concatenate elements of $\Omega$ extends to probability
measures and kernels on $\Omega$. We say that a measure $\mu \in
\prob(\Omega)$ and a kernel
$\nu \in \kernel(\Omega)$ on a TC-space are \define{compatible} at
the stopping time $\tau$ if
\[ \nlt_\omega(\sC_{\omega, \tau(\omega)})= 1, \text{ for
$\mu$-almost all $\omega$.
}\]

When $\ast$ factors through a state map $X$, a sufficient condition for
compatibility of $\mu\in\prob(\Omega)$ and $\nu\in\kernel(\Omega)$ at
$\tau$ is that
  \begin{align}
  \label{equ:crit-comp}
      \nlt_{\omega} \Big(X_0= X_{\tau}(\omega)\Big) = 1, \text{ for $\mu$-almost
    all $\omega$ with } \tau(\omega)<\infty.
  \end{align}
Using the convention, as above, that $\Omega\times \set{\infty} \times
\Omega' \subseteq \sC$,
we also note that, given a stopping time $\tau$, the set $\sC_{\tau} =
\sets{(\omega,\omega')}{ (\omega,\tau(\omega),\omega') \in \sC }$ is
a pullback of the Borel set $\sC$
via the measurable map $(\omega,\omega')
\mapsto (\omega,\tau(\omega),\omega')$, and, therefore, itself
measurable.

For $\mu \in \prob(\Omega)$ and a $\tau$-compatible
kernel $\nu\in\kernel(\Omega)$
let $\mu \otimes \nlt \in \prob(\Omega\times\Omega)$
denote the product of $\mu$ and the $\tau$-restriction of $\nu$.
The \define{concatenation} $\mu\ast_
{\tau} \nu$ is then defined as the push-forward of this product
via the measurable map $C_{\tau} \ni (\omega,\omega') \mapsto \omega \ast_{\tau(\omega)}
\omega'$. We note that the compatibility relation introduced above implies
that $\mu\otimes \nlt (C_{\tau})=1$, so that $\mu\at \nu$ is, indeed, a
probability measure. Moreover, we have
\begin{align*}
 	\int G(\omega)\, (\mu\at \nu)(d\omega)
 	&= \int G(\omega \at
 	\omega')\, (\mu\otimes\nlt)(d\omega,d\omega')
 	 = \iint G(\omega \at
 	\omega')\, \nlt_\omega
        (d\omega')\, \mu(d\omega),
\end{align*}
for any sufficiently integrable random variable $G$ on $\Omega$.
The compatibility condition \eqref{equ:TC-1} implies further that
\begin{equation}
\label{equ:act-ast}
\begin{split}
  \int G\, d(\mu\at \nu) &=
  \iint G(\omega_{\leq \tau} \at \omega')\,
  \nlt_\omega(d\omega')
  \mu(d\omega)\\ &=
  \iint G(\tomega \at \omega')\,
  \nlt_\omega(d\omega')
  \mlt(d\tomega),
\end{split}
\end{equation}
where $\mu_{\leq \tau}$ is the push forward of $\mu$ via $T_{\tau}$.

\subsubsection{Tail maps} Tail maps on TC-spaces will play an important
role in the dynamic programming principle and will model payoffs associated
to controlled processes.
\begin{definition}
	A measurable map $G$ from a TC-space to a measurable
	space $S$ is called a \define{tail map} if
	$G(\omega\ast_{t} \omega') = G(\omega')$ for all $t\in \Time$,
  all $\omega \in \Omega$ and all $\omega' \in \sC_{\omega,t}$.
  When $S=\R$ ($S=\bar{\R}$), a tail map is called a \define{tail random
  variable} (\define{extended tail random variable}).
\end{definition}
The tail property of random variables extends readily to stopping times in
the following form:
\begin{align*}
  G(\omega\at \omega') = \begin{cases}
  G(\omega'), &  \tau
  (\omega)<\infty \\ G(\omega), &  \tau(\omega)=\infty,
  \end{cases}
\end{align*}
as long as $\omega'$ is compatible with $\omega$ at $\tau$.
Combining this expression with \eqref{equ:act-ast} we obtain the following
equality, valid for each stopping time $\tau$, probability
$\mu\in\prob(\Omega)$, a $\tau$-compatible kernel $\nu\in\kernel(\Omega)$,
and a sufficiently integrable tail random variable $G$:
\begin{equation}
\label{equ:int-cat-G}
\begin{split}
\int G\, d(\mu\at \nu) =
\int \tilde{G}(\omega_{\leq \tau})\, \mu(d\omega),
\end{split}
\end{equation}
	where
\[ \tilde{G}(\omega) = G(\omega) \inds{\tau(\omega)=\infty}+
	\int G(\omega')\, \nu_{\omega}(d\omega')
	 \inds{\tau(\omega)<\infty}.\]

\subsection{Control correspondences} \

A map $f:A \to 2^B$, where $2^B$
denotes the power-set of $B$ is called a
\define{correspondence}
from $A$ to $B$, and is also denoted by $f:A \tto B$.  Its \define{graph}
$\Gamma(f) \subseteq A \times B$ is given by $\Gamma(f) = \sets{(a,b)}{a\in
A, b\in f(a)}$, and its \define{image} by $\image(f)=\cup_{a\in A} f(a)$.
A correspondence is said to be \define{non-empty-valued} if $f(a) \ne
\emptyset$ for all $a\in A$.

\begin{definition}
A non-empty-valued
correspondence $\sP:\Omega \tto \prob(\Omega)$, on a measurable space
$\Omega$ is called a \define{control correspondence}.
\end{definition}

Given a control correspondence $\sP$, a universally measurable random variable
$G$ is said to be \define{$\sP$-upper
semi-integrable}, denoted by $G\in \slmz(\sP)$, if $G^+ \in\slone(\mu)$ for
each $\mu \in \image \sP$.
To each control correspondence $\sP$ and each $G\in
\slmz(\sP)$  we associate the
\define{value function} $v: \Omega \to [-\infty,\infty]$,
given by
\begin{align}
\label{equ:value}
	v(\omega) = \sup_{\mu\in\sP(\omega)}
        \int G\, d\mu.
\end{align}

\subsection{Three key properties}\

There are three key properties that control correspondences must satisfy in
order for our main results to apply. These properties appear in
\cite{Zit14} in a similar terminologial setting, but have been considered and
understood in the literature in diffferent forms 
long before that (see \cite{ElKTan13,NutHan13} for two recent formulations).
We recall that a \define{universally measurable $\sP$-selector} (or,
simply, a \define{$\sP$-selector})
is a (universally measurable) kernel form $\Omega$ to $\prob(\Omega)$ with
the property that $\nu(\omega) \in \sP(\omega)$, for each
$\omega$; the family of all $\sP$-selectors is denoted by $\sS(\sP)$.
\begin{definition}
\label{def:analytic}
A control correspondence $\sP$ on standard Borel space $\Omega$ is called
\begin{enumerate}
\item \define{analytic} if its
graph $\Gamma(\sP)$ is an analytic
subset of the (standard Borel) space $\Omega\times \prob(\Omega)$.
\end{enumerate}
A control correspondence $\sP$ defined on a TC space $\filt$ is said to be
\begin{enumerate}[resume]
\item \define{concatenable}
if for each $\omega
\in \Omega$, $\mu\in \sP(\omega)$,  $\nu\in \sS(\sP)$, and each stopping
time $\tau$, $\nu$ is $\tau$-compatible with $\mu$ and
\begin{align*}
 \mu \at \nu \in \sP(\omega).
\end{align*}
\item \define{disintegrable} if for
each $\omega\in\Omega$, $\mu \in \sP(\omega)$
and a stopping time $\tau$ there exists $\nu \in \sS(\sP)$ such that
$\nu$ is $\mu$-compatible at $\tau$ and
\begin{align*}
  \mu = \mu \at \nu.
\end{align*}
\end{enumerate}
\end{definition}

\begin{remark}
\label{rem:unions}
It follows directly from the definitions of analyticity, concatenability
and disintegrability that the following, useful, implications hold for
any sequence of control correspondences $\seq{\sP}$
on the same Borel space $\Omega$.  Let $\cap_n \sP_n$ and
$\cup_n \sP_n$ be the intersection and the union, defined pointwise, on
$\seq{\sP}$.
\begin{enumerate}
  \item If each $\sP_n$ is analytic, then so are $\cup_n \sP_n$
  and $\cap_n \sP_n$.
  \item If each $\sP_n$ is concatenable, then so is $\cap_n
  \sP_n$.
  \item If each $\sP_n$ is disintegrable, then so is $\cup_n
  \sP_n$.
\end{enumerate}

\end{remark}

We state for completeness the following result which will be used in the
sequel, and the proof of which follows
almost verbatim the argument in \cite[Theorem 2.4, part 1.,
p.~1605]{Zit14}, which, in turn, is a reformulation of the standard
argument available, for example, in \cite{BerShr78}.
We remind the
reader of the convention $+\infty - \eps = 1/\eps$, for $\eps > 0$.
\begin{proposition}[Universal measurability of value functions]
\label{pro:DPP}
Suppose that $\Omega$ is a standard Borel space, $\sP$ an analytic control
correspondence, $G \in \slmz(\sP)$ and
that $v$ is
the associated value function, given by \eqref{equ:value}. Then
$v$ is universally measurable
and for each $\eps>0$ there exists a (universally measurable) selector $\nu^{\eps}\in \sS(\sP)$ such that
\begin{align*}
  v(\omega) -\eps \leq \int G\, d\nu^{\eps}_{\omega},\ \eforall
\omega\in\Omega.
\end{align*}
\end{proposition}

\subsection{An abstract version of the dynamic programming principle}\

We are ready to state the most abstract version of the DPP that holds in
our setting. A more directly applicable - and more familiar-looking -  version,
based on the notion of a state map will be given below.
The ideas in the proof are entirely standard. In fact, our setting is
constructed as the most flexible one where this proof can be applied.
We provide the details
for the reader's convenience.
\begin{theorem}[DPP]
\label{thm:DPP}
Let $\sP$ be an analytic control correspondence on a TC
space $\Omega$, $G \in \slmz(\sP)$ a
tail
random variable,
and $v$ the associated value function, given by
\eqref{equ:value}. Then,
\begin{enumerate}
\item If $\sP$ is concatenable, then for each $\omega\in\Omega$ and each
  stopping time $\tau$ we have
   \begin{align}
   \label{equ:DPP-conc}
      v(\omega) \geq \sup_{\mu \in \sP(\omega)}
      \int
      \Big( v\circ \Ttau \inds{\tau<\infty} + G \inds{\tau=\infty} \Big)\, d\mu
   \end{align}
\item If $\sP$ is disintegrable, then for each $\omega\in\Omega$ and each
  stopping time $\tau$ we have
   \begin{align}
   \label{equ:DPP-dist}
      v(\omega) \leq \sup_{\mu \in \sP(\omega)}
      \int  \Big(
      v\circ \Ttau \inds{\tau<\infty} + G \inds{\tau=\infty} \Big)
      \, d\mu
   \end{align}
\end{enumerate}
\end{theorem}
\begin{proof}
 Suppose, first, that $\sP$ is concatenable and pick $\omega \in \Omega$,
 $\mu \in \sP(\omega)$ and a stopping time $\tau$. Given $\eps>0$,
 Proposition \ref{pro:DPP} guarantees the existence of an $\eps$-optimizing
 selector $\nu^{\eps}$, i.e., such that $v^{\eps}(\omega):=\int G\, d\nu^{\eps}_{\omega} \geq
 v(\omega) - \eps$, for each $\omega \in \Omega$.
 We construct the measure $\mu'$ by concatenating $\mu$ and $\nu^{\eps}$ at
 $\tau$. The assumption of concatenability implies that they are
 compatible and that $\mu' \in \sP(\omega)$. Therefore,
  \begin{align*}
     v(\omega)  &\geq \int G\, d\mu' = \int G\, d (\mu \at \nu^{\eps}) =
    \iint G(\omega \at \omega')\, \nelt_\omega(d\omega')\,
    \mu(d\omega)\\
    &=
    \iint \Big(G(\omega) \inds{\tau(\omega)=\infty} +
    G(\omega')\inds{\tau(\omega)<\infty}\Big) \,
    \nelt_\omega(d\omega')\,  \mu(d\omega) \\
    &\geq \int \Big( G(\omega) \inds{\tau(\omega)=\infty} +
    (v(\omega_{\leq \tau})-\eps)
    \inds{\tau(\omega)<\infty}\Big) \, \mu(d\omega),
  \end{align*}
  which implies \eqref{equ:DPP-conc}.

  In the disintegrable case, we pick $\eps>0$, $\omega\in\Omega$,
  $\tau\in\Stop$ and choose $\mu^{\eps} \in \sP(\omega)$ such that
  $v(\omega) -\eps \leq \int G\, d\mu^{\eps}$. By disintegrability, we can
  write $\mu^{\eps} = \mu^{\eps} \at \nu$ for some $\nu \in \sS(\sP)$, and so \begin{align*}
    v(\omega) - \eps &\leq
     \int G\, d(\mu^{\eps}
    \at \nu) = \int \left( G(\omega) \inds{\tau=\infty} + \inds{\tau<\infty}\Big(
    \int G(\omega') \nlt_\omega (d\omega') \Big)\right) \,
    \mu(d\omega)\\
    &\leq \int \Big( G(\omega) \inds{\tau=\infty} + v(\omega_{\leq \tau})
    \inds{\tau<\infty}\Big) \,\mu(d\omega). \qedhere
  \end{align*}
\end{proof}

\subsubsection{State maps and factoring}
We remind the reader that, as defined in subsection \ref{sss:state}, a
state map $X:\Omega\to E$ is simply a measurable map from a T-space to a
Polish space $E$, and that $X_{\tau}$ is a shortcut for $X \circ T_{\tau}$,
for $\tau\in\Stop$. Just like (concatenation) compatibility may factor through
$X$, so can a control correspondence:
\begin{definition}
\label{equ:state-map}
A control correspondence $\sP$ on $\Omega$ is said to \define{factor through}
a state map $X$ if
there exists a correspondence $\bar{\sP}:E \tto \prob(\Omega)$ such that
$\sP (\omega) = \bar{\sP}( X(\omega) ) \subseteq \prob(\Omega)$, i.e., the following diagram commutes:
\begin{equation}
\label{equ:CD-state}
 \begin{tikzcd}
\Omega \arrow[r, "X"] \arrow[d,twoheadrightarrow, "\sP"]
& E \arrow[dl, twoheadrightarrow,  "\bar{\sP}"] \\
\prob(\Omega)
&
\end{tikzcd}
\end{equation}
\end{definition}
A very simple, but important, consequence of the existence of a state map
through which the
control correspondence $\sP$ factors is that in that case,
$v$ factors through it as well. Indeed, the function
$\bar{v}:E \to [-\infty,\infty]$, given by
$\bar{v}(x) = \sup_{ \mu \in \bar{\sP}(x)} \int G\, d\mu$, then has the
property that
$\bar{v}(X(\omega)) = v(\omega)$ and, under the conditions of Theorem
\ref{thm:DPP}, satisfies
\begin{align*}
	\bar{v}(x)
	\leq (\geq) \sup_{\mu\in\sP(x)} \int  \Big(
        \bar{v}(X_{\tau}) \inds{\tau<\infty} + G \inds{\tau=\infty} \Big)\,
        d\mu
\end{align*}
for all $x \in \image X$, and all stopping times $\tau\in \Stop$.

\section{Martingale-generated control correspondences}
\label{sec:mart-gen}
Our next task is so take the abstraction level down a notch and study a
class of control correspondences defined via a family of martingale conditions.
These correspondences generalize the standard martingale formulation in the
theory of stochastic optimal control  and are defined via a family of
structure-preserving maps into the model space space $\dr^0$ of $\R$-valued
\cd paths $x:\Time\to\R$ with $x(0)=0$.

\subsection{Canonical local martingale measures}\
\label{sss:local-mart-meas}

With the $T$-space structure of $\dr$ described in subsection
\ref{sss:exam-T},
each non-anticipating map $F$ from a T-space $\filt$ into $\dr$ induces a sequence
$\set{F^n}$ of non-anticipating maps
\begin{align}
\label{equ:Fn}
F^n_t  =
F_{\tau_n^F\wedge t} \ewhere
  \tau_n^F(\omega) = \inf\{t\geq 0: |F_t(\omega)| \geq n \} \wedge n.
\end{align}
When the choice of $F$ is evident from context,
we may drop the superscript and write $\tau_n=\tau^F_n$.
\begin{definition}
A probability measure $\mu \in \prob(\Omega)$ is said to be a
\define{canonical local-martingale probability for $F$} if
the stochastic process $\prst{F^n_t(\cdot)}$
is a martingale under ($\mu,\FFF$) for each $n\in\N$.
The set of all canonical local martingale probabilities for $F$ is denoted by
$\sM^{F,loc}$.
\end{definition}
\begin{remark}
\label{rem:1252}
The notion of a \emph{canonical} local martingale differs from the standard
notion of a local martingale in that
it requires that the reducing sequence takes a particular form, namely that
of the sequence of space-time
exit times.  This requirement is nontrivial, as it is known that there are
local martingales that cannot be reduced by this particular sequence (see
\cite[Lemme 2.1., p.~57]{Str77}). On
the other hand, this notion suffices for many applications; indeed
for continuous processes (or processes with jumps bounded from below) the notions of a
  canonical local martingale and that of a local martingale coincide.
\end{remark}
With the notion of a canonical local martingale probability under our belt,
we can define a large class of control correspondences. Housed on T-spaces, they
need two ingredients to be specified: 1) a family of $\sD$ of
non-anticipating maps from $\Omega \to \dr$, and 2) a state map $X$ from
$\Omega$ to a Polish space $E$. Once these are specified,
for $x\in E$ we define
\begin{align}
\label{equ:spx}
   \bsP(x) = \bigcap_{F\in \sD} \sM^{F,loc} \cap \Bsets{ \mu \in \prob(\Omega)}{
     X_0=x,\text{ $\mu$-a.s.}},
\end{align}
where, as usual, $X_0$ is the shortcut for $X\circ T_0$.
The  $(\sD,X)$-\define{generated control correspondence} $\sP=\sP(\sD,X):
\Omega \tto \prob(\Omega)$ is then defined by
\[ \sP(\omega) = \bsP(X(\omega)) \efor \omega\in\Omega,\]
so that it naturally factors through $X$.
\subsection{Sufficient conditions for analyticity}\

The ubiquitous Polish-space structure woven
into all the ingredients of our setup makes it possible to give widely met
sufficient conditions on the family $\sD$ such that the resulting
$(\sD,X)$-correspondence becomes analytic. The countability condition we
impose on $\sD$ is not the
weakest possible, but since it holds in most relevant examples, we only
comment on some possible routes towards establishing weaker versions
in Remark \ref{rem:weak} below.
\begin{proposition}
\label{pro:analy-mart}
Let $\sD$ be a countable family
of non-anticipating maps from a T-space
$\Omega$ to $\dr$ and let $X:\Omega\to E$ be a state map.
Then the $(\sD,X)$-generated control correspondence $\sP$ is analytic.
\end{proposition}
The proof is based on a modification of
\cite[Lemma 3.6, p.~1611]{Zit14}, where
\begin{align*}
  \Qstop = \Bsets{ q \ind{A} + r \ind{A^c} }{ q\leq r\in\Qtime, A\in\Pi_q }
\end{align*}
with $\Qtime$ denoting a countable dense set in $\Time$,
and $\{\Pi_q\}_{q\in\Qtime}$ a collection of
countable $\pi$-systems such that $\sigma(\Pi_q)=\sF_q$ for all $q\in\Qtime$.
The exact choice of $\Qtime$ or $\{\Pi_q\}_{q\in\Qtime}$ is unimportant, as long as it is
fixed throughout.
\begin{lemma}
\label{lem:F-mart}
For each non-anticipating map $F$, we have
\begin{align}
\label{equ:mart-count}
   \sM^{F,loc} = \bigcap
   \Bsets{
     \mu\in\prob(\Omega) }{ F^n_q, F^n_r\in\lone(\mu) \eand
     \EE^{\mu}[F^n_r \ind{A}] = \EE^{\mu}[ F^n_q
     \ind{A} ] }
\end{align}
where the intersection is taken over all
$n\in\N$,
   $q<r \in \Qtime$ and $A\in \Pi_q$.
\end{lemma}
\begin{proof}
The inclusion $\sM^{F,loc} \subseteq \dots$ is straightforward. Conversely,
let $\mu \in \prob(\Omega)$ be an element of the right-hand side of
\eqref{equ:mart-count}.
We first show that $\mu\in\sM^{F^n}_{\Qtime}$, where
$\sM^{F^n}_{\Qtime}$
denotes the set
of all $\mu\in\prob(\Omega)$ with the property
that $\fml{F^n_t}{t\in\Qtime}$ is a $\mu$-martingale
with respect to
$\fml{\sF_t}{t\in\Qtime}$.
That is an immediate consequence of the equalities of expectations under
$\mu$ on the right-hand-side of \eqref{equ:mart-count}. Considered over all
$A\in \Pi_q$, with $q<r\in\Qtime$, they amount to
    $\EE^{\mu}[ F^n_r | \sF_q] = F^n_q$, a.s., by $\pi$-$\ld$-theorem.

It remains to argue that $F^n$ is a $\mu$-martingale on entire $\Time$. Assuming, without
loss of generality, that $\Time=[0,\infty)$,
we start by picking
$s\in \Time \setminus \Qtime$
and $r\in \Qtime$ with $r>s$.
The backward martingale convergence theorem implies that
\[ \EE^{\mu}[ F^n_r| \sF_{s+}] = F^n_s, \text{ $\mu$-a.s.}\]
Since $F^n$ is non-anticipating,
$F^n_s$ is $\sF_s$-measurable and we may replace $\sF_{s+}$ by $\sF_{s}$
in the equality above.
Finally, for $t \in \Time$ with $t>s$,  we approximate $F^n_t$ by a sequence
$\sqm{F^n_{r_m}}$ with
$r_m\downto t$ and $r_m\in \Qtime$, to conclude that
$F^n$ is, indeed, a martingale under $\mu$.
\end{proof}
\begin{proof}[Proof of Proposition \ref{pro:analy-mart}]
For each $r\in\Time$,
the coordinate maps are Borel measurable on $\dr$
and, so,
$\mu\mapsto \EE^{\mu}[ F_r \ind{A}]$ is Borel on $\Omega$. It is easy to
see that the family of probability measures under which a given real-valued Borel map is
integrable is also a Borel set, so it follows that
$\sM^{F,loc}$ is Borel for each $F$. The countability of $\sD$ guarantees
that $\cap_{F\in \sD} \sM^{F,loc}$, as well. Finally, the graph of $\sP$ is
analytic (in fact Borel) as it is given as an
intersection of Borel sets
\[
\Gamma(\sP) = \Bsets{ (\omega,\mu) }{ \mu\Big(X_0 = X_0(\omega)
\Big)=1 }
\cap \left( \Omega \times
\bigcap_{F\in\sD}\sM^F \right). \qedhere\]
\end{proof}
\begin{remark}
\label{rem:weak}
When $\sD$ is not countable, the set $\cap_{F\in \sD} \sM^{F,loc}$ is not
necessarily Borel measurable (or even analytic) in general. The situation is somewhat more pleasant when $\sD$ admits a
structure of a Borel space with the property that the maps
\[ \sD \ni F \mapsto \EE^{\mu}[F_r], r\in \Time,\]
are measurable for each probability measure $\mu\in\prob(\Omega)$. In that
case, the intersection $\cap_{F\in\sD} \sM^{F,loc}$ can be represented as a
\emph{co-projection}
\[ \sets{\mu\in\prob(\Omega)}{ \forall\, F \in \sD,\ (F,\mu)\in \sM} \]
of the Borel set
$\sM =\sets{ (F,\mu) \in \sD \times \prob(\Omega) }{ \mu \in \sM^{F,loc} }$.
Unlike projections, the images of co-projections are co-analytic, but not
necessarily analytic sets.
Not everything is lost, however,
as we usually know a great deal more about the set
$\sM$, other than the fact that it is a Borel set. Indeed, the countable
case of Proposition \eqref{pro:analy-mart} corresponds to the
measurable-selection theorem of Lusin for sets with countable sections (see
\cite[Theorem 5.7.2, p.~205]{Sri98}). On the other side of the spectrum are
measurable-selection theorems with large sections (see Section 5.8 in
\cite{Sri98}), which can be used for certain uncountable $\sD$.
\end{remark}

\subsection{Sufficient conditions for concatenability}\
\label{sse:suff-con}
Having discussed analyticity, we turn to the second major assumption of our
abstract DPP theorem, namely concatenability. It is not hard to see that
without additional requirements on $\sD$, no $(\sD,X)$-generated
control correspondence should be expected to be concatenable. A natural
requirement, as we will see below, is that the maps $F$ be  TC-morphisms,
introduced in Definition \ref{def:TC-morph} above.
Moreover, the target space for these
TC-morphisms will be $\dr^0$ - a model space for (the laws of) local
martingales.
We remind the reader
(see section \ref{sss:exam-TC} above) that $\dr$ comes with two different
natural concatenations, namely the strict one ($\bullet$) and the
adjusted one ($\star$).
We will only work with the adjusted one in this
section, but, in order to avoid any confusion, we
will write $(\dr,\star)$ and $(\dr^0,\star)$ throughout.

\begin{definition}
A map $F:\Omega \to \dr$ is said to be \define{canonically locally bounded}
if there exists a sequence $\seq{M}$ of positive constants so that
 \begin{align}
 \label{equ:Mn}
   \abs{F^n(\omega)_t} \leq M_n \eforall \omega\in\Omega, t\in \Time.
 \end{align}
A simple sufficient condition for canonical local boundedness is that the
jumps of $F$ (when seen as a stochastic process on $\Omega$) are uniformly bounded.
\end{definition}

\begin{proposition}
\label{pro:concat-mart}
Let $\sD$ be a family of canonically locally bounded TC-morphisms into
$(\dr^0,\star)$, and let $X$ be a state map.
Then the $(\sD,X)$-generated control correspondence $\sP$ is closed under concatenation.
\end{proposition}
The proof is based on the several lemmas.
We omit the straightforward proof of the first one.
\begin{lemma}
\label{lem:increment}
Suppose that $F$ is a $TC$-morphism into $(\dr,\star)$.
For all stopping times $\kappa$ we have
\begin{align*}
  F_{\kappa+t}(\omega*_{\kappa}\omega') - F_{\kappa+s}(\omega*_{\kappa}\omega')
  = F_t(\omega') - F_s(\omega')
\end{align*}
for all $\omega\in\Omega$ with $\kappa(\omega)<\infty$,
 $\omega'\in\sC_{\omega,\kappa(\omega)}$ and
all $s,t\in\Time$.
\end{lemma}
Our second lemma gives a convenient characterization of
canonical local martingales.
We use $\Stop$, as in the case of
T-spaces, to denote the set of all $\Time$-valued (raw) stopping times.
We also write $Y^n = Y^{\tau_n}$, where
$\tau_n = \inf\sets{t\geq 0}{\abs{Y_t} \geq n} \wedge n$, and note that all
sampled values of $Y$ in the statement are well-defined thanks to the fact
that each $Y^n$ is constant after $t=n$.
\begin{lemma}
\label{lem:mart-char}
Let $(\Omega,\sF,\FFF=\prst{\sF_t},\PP)$ be a filtered probability space,
$\prst{Y_t}$ a \cd and adapted process,
and $\kappa$ a stopping time with $Y^n_{\kappa}\in\lone$ for each $n\in\N$.
Then, the following two statements are equivalent
\begin{enumerate}
\item $Y$ is a canonical local martingale.
\item $G\in\lone$ and $\EE[G]=0$ for all
\[ G\in \bigcup_{n\in\N}\sX_n^{\leq \kappa}(Y) \cup \sX_n^{\geq
\kappa}(Y),\] where
the countable sets $\sX_n^{\leq \kappa}$ and  $\sX_n^{\geq \kappa}$
are given by
\begin{align*}
  \sX_n^{\leq \kappa}(Y)
  &= \Bsets{ Y^n_{\tau\wedge\kappa} - Y^n_{\kappa} }{\tau\in\Qstop}, \\
  \sX_n^{\geq \kappa}(Y)
  &= \Bsets{ Y^n_{\tau\vee\kappa} - Y^n_{\kappa} }{\tau\in\Qstop}.
\end{align*}
\end{enumerate}
\end{lemma}
\begin{proof}
$(1) \Implies (2)$\
Assuming that $Y$ is a canonical local martingale, each $Y^n$ is
martingale constant after $t=n$, and therefore a uniformly-integrable martingale. Stopping times in $\Qstop$ are bounded, so, by the optional
sampling theorem, (2) holds.

$(2) \Implies (1)$\ Suppose that (2) holds and that $n\in\N$ is fixed.
We take the advantage of the fact that $Y$ is
\cd to conclude (as in the proof of Lemma \ref{lem:F-mart}) that
it suffices
to show that $Y^n$ is a martingale on $\Qtime$. For that, in turn,
we choose $\tau\in\Qstop$, so that
$\tau = p\ind{A} + q\ind{A^c}$ for some
$p\leq q\in\Qtime$ and $A\in\Pi_p$ and note that
\begin{align*}
   Y^n_{\tau} - Y^n_\kappa =
  \Big(Y^n_{\tau\wedge\kappa} - Y^n_\kappa\Big)
  + \Big(Y^n_{\tau\vee\kappa} - Y^n_\kappa \Big).
\end{align*}
Since $Y^n_{\tau\wedge\kappa} - Y^n_{\kappa} \in \sX^{\leq \kappa}$
$Y^n_{\tau\vee\kappa} - Y^n_{\kappa} \in \sX^{\geq \kappa}$ and
$Y^n_{\kappa} \in \lone$, we conclude that $ Y^n_{\tau}
\in\lone$ and
that $\EE[ Y^n_{\tau}] = \EE[ Y^n_\kappa] $.
It follows that the value of
$\EE[ Y^n_{\tau}]$ does not depend on the choice of $\tau$, making $Y^n$
into a martingale.
\end{proof}

\begin{lemma}
\label{lem:tau-split}
Let $\Omega$ be a TC-space and $\kappa,\tau\in\Stop$  such that $\kappa\leq\tau$.
For $\omega\in\Omega$ we define $\tau'_{\omega}$ by
$$ \tau_\omega'(\omega') = \begin{cases}
\tau(\omega \ast_\kappa \omega') - \kappa(\omega),
& \kappa(\omega)<\infty \eand \omega'\in\sC_{\omega,\kappa(\omega)} \\
+ \infty, & \text{otherwise,} \\ \end{cases} $$
Then the map $(\omega,\omega')\mapsto\tau_\omega'(\omega')$
is jointly measurable,
$\tau_\omega'\in\Stop$ for any fixed $\omega\in\Omega$,
and $\tau(\omega*_\kappa\omega') = \kappa(\omega) + \tau_\omega'(\omega')$.
\end{lemma}
\begin{proof}
By construction, we clearly have
$\tau(\omega*_\kappa\omega') = \kappa(\omega) + \tau_\omega'(\omega')$.
With the convention that $\tau(\omega\ak \omega') - \kappa(\omega)=\infty$
when $\kappa(\omega)=\infty$, we note that $\tau'$ can be expressed as:
\begin{align*}
  \tau_\omega'(\omega')
  &= (+\infty) \ind{\sC^c}(\omega,\kappa(\omega),\omega')
  + (\tau(\omega \ast_\kappa \omega') - \kappa(\omega))
  \ind{\sC}(\omega,\kappa(\omega),\omega')
\end{align*}
and is hence jointly measurable.
It remains to argue that $\tau_\omega'$ is a stopping time.
We fix $\omega\in\Omega$ with $k=\kappa(\omega)<\infty$, and for
 $s \in \Time$ define
\[ A=\sets{\omega'\in\Omega}{ \tau'(\omega') \leq s } = \sets{ \omega' \in
\sC_{\omega,k}}{
  \tau(\omega \ast_k \omega') \leq s+k }.\]
By Proposition \ref{pro:prop-trunc}, part (1), it will suffice to show that
$\Ts^{-1}(A)=A$, i.e., for $\omega'\in\Omega$ we have $(a) \Leftrightarrow (b)$, where
\begin{enumerate}[label=$(\alph*)$]
  \item
$ \omega' \in \sC_{\omega,k}  \text{ and } \tau(\omega\ast_k \omega')\leq
s+k$, and
\item
$ (\omega')_{\leq s} \in \sC_{\omega,k}  \text{ and } \tau(\omega\ast_k
(\omega'_{\leq s})) \leq
s+k$.
\end{enumerate}
The first, compatibility-related, parts of statements of (a) and (b) are equivalent to each
other by
the assumptions in
\eqref{equ:TC-comp} of Definition \ref{def:TC}. To deal with the
inequalities involving $\tau$ we use
Proposition
\ref{pro:prop-trunc}, part (2), as well as the assumption \ref{equ:TC-2} of
Definition \ref{def:TC} to conclude that
 \begin{align*}
   \tau\Big( \omega\ast_k (\omega'_{\leq s})  \Big) \leq s+k
   &\Leftrightarrow
   \tau\Big( \big(\omega\ast_k (\omega'_{\leq s}) \big)_{\leq s+k}  \Big) \leq s+k
   \Leftrightarrow
   \tau\Big( (\omega\ast_k \omega')_{\leq s+k}  \Big) \leq s+k\\
   &\Leftrightarrow
   \tau\Big( \omega\ast_k \omega'  \Big) \leq s+k. \qedhere
 \end{align*}
\end{proof}

\begin{proof}[Proof of Proposition \ref{pro:concat-mart}]
Let $\sP$ be the $(\sD,X)$-generated control correspondence as in the statement,
and let
$\omega_0\in \Omega$, $\mu\in \sP(\omega_0)$,  a kernel $\nu \in
\sS(\sP)$ and a stopping time $\kappa$ be given.

First, we argue that $\nu$ is $\kappa$-compatible with  $\mu$.
By the definition of $\sP$, we have $\nu_{\omega}( X_0 = X(\omega) )=1$ for
each $\omega\in\Omega$. After a composition with $T_{\kappa}$, we get
$\nu^{\leq\kappa}_{\omega}(X_0 = X_{\kappa}(\omega)) = 1$ for each $\omega\in\Omega$,
which implies compatibility,
according to the criterion of \eqref{equ:crit-comp}.

Next, we show that $\mu'=\mu\ak\nu\in\sP(\omega_0)$.
Part (2) of Definition \ref{def:TC} makes it clear that for $x=X_0(\omega_0)$
we have $\mu'(X_0=x)=1$.  Therefore, we need to argue that
$\mu'\in\sM^{F,loc}$, for each $F\in\sD$.
By Lemma \ref{lem:mart-char}, this is equivalent to checking
$\int G\, d(\mu\ak\nu) = 0 $ for all
$G \in \cup_{n\in\N}\sX_n^{\leq \kappa}(F)\cup \sX_n^{\geq \kappa}(F)$.
We fix $n\in\N$ and treat the two cases separately:

1. $G \in \sX_n^{\leq \kappa}(F)$:\
In this case there exists $\tau\in\Qstop$,
such that
$G(\omega) = F^n_{(\tau\wedge\kappa)(\omega)}(\omega)
- F^n_{\kappa(\omega)}(\omega)$.
By Definition \ref{def:TC}, part (2),
we have $(\tau\wedge\kappa)(\omega\ak \omega') =
(\tau\wedge\kappa)(\omega)$
and $\kappa(\omega\ak\omega') = \kappa(\omega)$,
so that, by the non-anticipativity of $F^n$
(which follows from the non-anticipativity of $F$),
we have
\begin{align*}
  G(\omega \ak \omega')
  = F^n_{(\tau\wedge\kappa)(\omega)}(\omega\ak \omega')
  - F^n_{\kappa(\omega)}(\omega \ak \omega')
  = F^n_{(\tau\wedge\kappa)(\omega)}(\omega)
  - F^n_{\kappa(\omega)}(\omega) = G(\omega).
\end{align*}
Since $G$ is bounded (since so is $F^n$) we have
\begin{align*}
  \int G\, d\mu' &= \iint G(\omega\ak \omega')\,
  \nu^{\leq \kappa}_\omega(d\omega')\, \mu(d\omega) = \int G(\omega)\, \mu(d\omega)=0,
\end{align*}
where the last equality follows from the fact that $\mu \in\sM^{F,loc}$.

2. $G \in \sX_n^{\geq \kappa}(F)$:\
Let $\tau\in\Qstop$ be such that $G=F^n_{\tau\vee\kappa}-F^n_\kappa$. Then
\begin{align*}
  \int F^n_{\tau\vee\kappa}(\omega) - F^n_\kappa(\omega) \,\mu'(d\omega)
  &=  \int \inds{\tau_n>\kappa}(\omega)
  (F^n_{\tau\vee\kappa}(\omega) - F^n_\kappa(\omega)) \,\mu'(d\omega) \\
  &= \int \inds{\tau_n>\kappa}(\omega)
  (F_{(\tau\wedge\tau_n)\vee\kappa}(\omega) - F_\kappa(\omega))\,\mu'(d\omega)
\end{align*}
Note that $(\tau\wedge\tau_n)\vee\kappa \geq \kappa$,
and let $\tau'$ be as in Lemma \ref{lem:tau-split}
(applied to $(\tau\wedge\tau_n)\vee\kappa$).
Also note that by Proposition \ref{pro:prop-trunc},
$\{\tau_n>\kappa\}\in\sF_\kappa=\sigma(T_\kappa)$.
Therefore $\inds{\tau_n>\kappa}$ is $\sigma(T_\kappa)$-measurable
and so $\inds{\tau_n>\kappa}(\omega\ak\omega')
=\inds{\tau_n>\kappa}(\omega_{\leq\kappa})=\inds{\tau_n>\kappa}(\omega)$.
Continuing with the equalities from above, we have
\begin{align*}
  &\int F^n_{\tau\vee\kappa}(\omega) - F^n_\kappa(\omega) \,\mu'(d\omega) =
  \\
  &= \iint \inds{\tau_n>\kappa}(\omega)
   (F_{\kappa(\omega)+\tau_\omega'(\omega')}(\omega\ak\omega')
  - F_\kappa(\omega\ak\omega'))\,\nu_\omega(d\omega')\mu(d\omega) \\
  &=
  \iint \inds{\tau_n>\kappa}(\omega)
   (F_{\tau_\omega'}(\omega')
  - F_0(\omega'))\,\nu_\omega(d\omega')\mu(d\omega)
  = \iint \inds{\tau_n>\kappa}(\omega)
   F_{\tau_\omega'}(\omega')\,\nu_\omega(d\omega')\mu(d\omega).
\end{align*}
where the last equality used the TC-morphism assumption together with
Lemma \ref{lem:increment}.
With $M_n$ given by \eqref{equ:Mn},
$\abs{F}$ is bounded
on $[0, \tau'_{\omega}]$
by $2 M_{2n}$
when
$\omega\in \set{\kappa<\tau_n}$,
By the canonical local martingale property, we have
  $\int F_{\tau_\omega'}(\omega')\,\nu_\omega(d\omega')$  for each $\omega
  \in \set{\kappa<\tau_n}$. Thanks to boundedness, again, the integral
 $\int G\, d\mu'$ can be computed as an iterated integral $ \int
 \inds{\tau_n>\kappa}(\omega) \int
   F_{\tau_\omega'}(\omega')\,\nu_\omega(d\omega')\mu(d\omega)$ and, so,
   $\int G\, d\mu'=0$.
\end{proof}

\subsection{Sufficient conditions for disintegrability}\

\subsubsection{Shift operators}
The key to disintegrability for martingale-generated control correspondences is the existence of a shift operator, as described below.
It plays the role of a partial inverse of the concatenation
operator in the second argument.
\begin{definition}
\label{def:shift}
A measurable map $\theta:\Time \times \Omega \to \Omega$ is said to be a
\define{shift operator} if for all $\omega\in\Omega$, $t,s\in
\Time$ and $\omega'\in\sC_{\omega,t}$
\begin{enumerate}
\item
$\theta_t(\omega) \in \sC_{\omega,t}$ and
    $\omega \ast_t \theta_t(\omega) = \omega$,
\item \label{ite:shift-adapt}
$(\theta_t(\omega))_{\leq t+s} = (\theta_t(\omega_{\leq s}))_{\leq t+s}$
\end{enumerate}
\end{definition}

\begin{remark}
Since $\sF_t = \sigma(T_t)$ on $\Omega$,
then part \eqref{ite:shift-adapt} of Definition \ref{def:shift} is
equivalent to the $(\sF_s,\sF_{t+s})$-measurability of $\theta_t$ for all
$t,s\in\Time$, i.e.,
\begin{align*}
  \forall\, t,s\in\Time:\quad
  \theta_t^{-1}(\sF_{t+s}) \subset \sF_s
\end{align*}
\end{remark}

The stopping-time version of a shift operator $\theta$ is defined in the natural way
\[ \theta_{\tau}(\omega) = \theta_{\tau(\omega)}(\omega),\]
where, for definiteness, we set $\theta_{\infty}(\omega) = \omega$, for all
$\omega$. This way, $\theta_{\tau}:\Omega\to\Omega$ is Borel measurable and
retains the property that
$\omega \at \theta_{\tau}(\omega) = \omega$, for all $\omega\in\Omega$ and
$\tau \in \Stop$.

\begin{lemma}
\label{lem:theta}
For any $\kappa,\sigma\in\Stop$, the following is also a stopping time:
\begin{align*}
  \tau(\omega) := \kappa(\omega) + \sigma(\theta_\kappa(\omega))
\end{align*}
\end{lemma}
\begin{proof}
Fix any $t\in\Time$ and $\omega\in\Omega$.
In order to show $\{\tau\leq t\}\in\sF_t$,
it is enough to show that
$\tau(\omega)\leq t$ if and only if $\tau(\omega_{\leq t})\leq t$.
Applying Proposition \ref{pro:prop-trunc} to $\sigma$
and using part (2) of the definition of $\theta$
gives the following equivalence:
\begin{align*}
  \tau(\omega)\leq t
  &\Leftrightarrow
  \sigma(\theta_{\kappa(\omega)}(\omega)) \leq t - \kappa(\omega) \\
  &\Leftrightarrow
  \sigma((\theta_{\kappa(\omega)}(\omega))_{\leq t-\kappa(\omega)})
    \leq t - \kappa(\omega) \\
  &\Leftrightarrow
  \sigma((\theta_{\kappa(\omega)}(\omega_{\leq t}))_{\leq t-\kappa(\omega)})
    \leq t - \kappa(\omega) \\
  &\Leftrightarrow
  \sigma(\theta_{\kappa(\omega)}(\omega_{\leq t}))
    \leq t - \kappa(\omega)
\end{align*}
First suppose $\tau(\omega)\leq t$.
Since $\kappa$ is a stopping time and $\kappa(\omega)\leq\tau(\omega)\leq t$,
then $\kappa(\omega)=\kappa(\omega_{\leq t})$.
Together with the above equivalence, this implies:
\begin{align*}
  \tau(\omega_{\leq t})
  &= \kappa(\omega_{\leq t})
  + \sigma(\theta_{\kappa(\omega_{\leq t})}(\omega_{\leq t})) \\
  &= \kappa(\omega) + \sigma(\theta_{\kappa(\omega)}(\omega_{\leq t}))
  \leq t
\end{align*}
For the other direction, suppose $\tau(\omega_{\leq t}) \leq t$.
Since $\kappa$ is a stopping time and
$\kappa(\omega_{\leq t})\leq\tau(\omega_{\leq t})\leq t$,
then $\kappa(\omega_{\leq t})=\kappa(\omega)$.
Therefore:
\begin{align*}
  \kappa(\omega) + \sigma(\theta_{\kappa(\omega)}(\omega_{\leq t}))
  &= \kappa(\omega_{\leq t})
  + \sigma(\theta_{\kappa(\omega_{\leq t})}(\omega_{\leq t}))
  = \tau(\omega_{\leq t})
  \leq t,
\end{align*}
which implies $\tau(\omega)\leq t$ by the equivalence above.
\end{proof}

\begin{proposition}
\label{pro:dis}
Let $\filt$ be a TC-space with concatenation operator $\ast$,
on which a shift operator $\theta$ is defined.
Suppose each $F\in \sD$ is a canonically locally bounded TC-morphism into $(\dr^0,\star)$,
and that $\ast$ is a factor of $X$.
Then, for each $\omega_0\in \Omega$, $\mu\in \sP(\sD,X)(\omega_0)$, and $\kappa\in
\Stop$ there exists a version $x\mapsto \bar{\nu}_x$ of the regular conditional
probability  $\mu(\theta_{\kappa} \in \cdot | X_{\kappa}=x)$ such that
for $\nu = \bar{\nu}\circ X$ we have
\[ \nu \in \sS(\sP) \eand \mu \ast_{\kappa} \nu = \mu.\]

In particular, $\sP(\sD,X)$ is disintegrable.
\end{proposition}
\begin{proof}
Having fixed a shift operator $\theta$, we pick $\omega_0\in\Omega, \mu\in
\sP(\omega_0)$ and $\kappa\in\Stop$.
For a stopping time $\sigma\in\Qstop$ and define
\begin{align*}
  \sigma_n(\omega) &= (\sigma\wedge\tau_n)(\omega) \\
  \tau(\omega) &= \kappa(\omega) + \sigma( \theta_{\kappa}(\omega) )
\end{align*}
so that $\tau$ is a stopping time by Lemma \ref{lem:theta}.
Since $F$ is a TC-morphism into  $(\dr^0,\star)$
Lemma \ref{lem:increment} implies that
\begin{align*}
  F_\tau(\omega) - F_\kappa(\omega)
  &= F_{\kappa+\sigma_n(\theta_\kappa)}(\omega\ak\theta_\kappa\omega)-
  F_{\kappa}(\omega)
  = F_{\sigma_n}(\theta_\kappa\omega)=
  F^n_{\sigma}(\theta_\kappa\omega).
\end{align*}
The same Lemma implies that $\abs{F}$ is bounded by $\abs{F_\kappa}+M_n$ on
the entire stochastic interval $[0,\tau]$. In particular, for
$A_m=\set{\abs{F_\kappa}\leq m}$ we have
\[
  \ind{A_m} F_{\sigma_n} \circ \theta_\kappa =
\ind{A_m} \Big(F_{\tau}-F_{\kappa}\Big)  = \ind{A_m} \Big( F^{m+M_n}_{\tau}
- F^{m+M_n}_{\kappa} \Big).\]
Since $F^{m+M_n}$ is a bounded martingale under $\mu$, for any
bounded measurable function $H$ on $E$ we have
   $\int H(X(\omega_{\leq\kappa}))  \ind{A_m}(\omega)
   F^n_{\sigma}(\theta_{\kappa}\omega) \,\mu(d\omega) =0$, and,
given that $F^n$ is bounded, we can
pass to the limit $m\to\infty$ by the dominated convergence theorem to
obtain
    \begin{align}
    \label{equ:H}
       \int
       H(X(\omega_{\leq\kappa})) F^n_{\sigma}(\theta_{\kappa}\omega)
      \,\mu(d\omega) =0,
    \end{align}
for all bounded and measurable $H$.
With $\bar{\nu}_x$ denoting a version of the regular conditional distribution of
$\theta_{\kappa}$ given $X_{\kappa}=x$,  we then have
\begin{align*}
 0 =  \int
       H(X(\omega_{\leq\kappa})) F^n_{\sigma}(\theta_{\kappa}\omega)
 \mu(d\omega) =
  \iint H(x) F^n_{\sigma}(\omega') \, \bar{\nu}_x (d\omega') \, \mu_{X_{\kappa}}(dx),
\end{align*}
where $\mu_{X_{\kappa}}$ is the $\mu$-distribution of $X_{\kappa}$. Since
$H$ is arbitrary, it follows that
 \begin{align}
 \label{equ:DX}
    \int F^n_{\sigma}\, d\bar{\nu}_x =0 \efor \mu_{X_\kappa}\text{-almost all
   }x\in E,
 \end{align}
for all $\sigma\in\Qstop$ and all $n\in\N$.
Since $\Qstop$ is countable,
there exists a set $\sN_1\in\Borel(E)$ such that
$\mu_{X_\kappa}(\sN_1)=0$, and the equality in \eqref{equ:DX}
holds for all $x\in E\setminus\sN_1$ and $\sigma\in\Qstop$.
Therefore $\bar{\nu}_x\in \sM^{F,loc}$ for all $x\in E\setminus\sN_1$.

Since $\ast$ is a factor of $X$, we have
$X( T_{\kappa}(\omega) ) = X_0(\theta_{\kappa}(\omega))$ for all $\omega$,
and so
\begin{align*}
 1=\int \inds{X_{\kappa}(\omega)= X_0(\ogk)}\, \mu(d\omega) =
 \iint \inds{x = X_0(\omega')} \bar{\nu}_{x}(d\omega') \mu_{X_{\kappa}}(dx),
\end{align*}
This implies that there exists another zero set $\sN_2\in\Borel(E)$
such that $\mu_{X_\kappa}(\sN_2)=0$ and $X_0=x$, $\bar{\nu}_x$-a.s.
for all $x\in E\setminus\sN_2$.
Hence, $\bar{\nu}_x \in \bar{\sP}(x)$
(where $\bar{\sP}(x)$ is defined in \eqref{equ:spx})
for all $x\notin\sN_1\cup\sN_2$.
By picking a selector $\bar{\nu}'$ of $\bar{\sP}$
(which is nonempty by Proposition \ref{pro:DPP})
and using it to set the values of
$\bar{\nu}_x$ on $\sN_1\cup\sN_2$, we can arrange that $\bar{\nu}_x\in
\bar{\sP}(x)$, for all $x\in E$.
\end{proof}

\subsection{The main result for martingale-generated control
correspondences}
\begin{theorem} [DPP for martingale-generated control correspondences]
\label{thm:DPP-mart}
Let $\filt$ be a TC-space with concatenation operator $\ast$ and a
shift operator $\theta$.
Suppose that
$X$ is a state map from $\Omega$ to a Polish space $E$
such that $\ast$ is a factor of $X$, and
that $\sD$ is a countable collection of
canonically locally bounded
TC-morphisms from $(\Omega,\ast)$ into $(\dr^0,\star)$.
Let $\sP=\sP(\sD,X)$, i.e.,
\begin{align*}
   \bsP(x) &= \bigcap_{F\in \sD} \sM^{F,loc} \cap \Bsets{ \mu \in \prob(\Omega)}{
     X_0=x,\text{ $\mu$-a.s.}} \\
   \sP(\omega) &= \bsP(X(\omega)) \text{ for } \omega\in\Omega,
\end{align*}
let $G\in\slmz(\sP)$ be a tail random variable,
and let the value function $\bar{v}$ be given by
\begin{align*}
  \bar v(x) &= \sup_{\mu\in\bsP(x)} \int G\, d\mu.
\end{align*}
Then for all $\omega\in\Omega$, $x\in E$, and $\tau\in\Stop$ we have:
\begin{align*}
  \bar v(x) &= \sup_{\mu \in \bsP(x)}
  \int \Big(\bar v (X_\tau) \inds{\tau<\infty} + G \inds{\tau=\infty}\Big) \, d\mu
\end{align*}
\end{theorem}
\begin{proof}
Use Propositions \ref{pro:analy-mart}, \ref{pro:concat-mart},
and \ref{pro:dis} to get the analyticity, concatenability, and disintegrability
(respectively) of the control correspondence $(\sD,X)$.
Then apply Theorem \ref{thm:DPP}.
\end{proof}

\section{Application 1 - Controlled Diffusions in the Weak Formulation}
\label{sec:app1}
\subsection{Problem formulation and the main result}\

Throughout this section we fix the following:
\begin{enumerate}
  \item a nonempty open set $\sO$ in $\R^n$ and set $E=\Cl \sO$ (the
  \define{state space}),
  \item a nonempty standard Borel space $A$, (the \define{control space}),
  \item Borel measurable functions $\beta:E\times A \to \R^n$ and
  $\sigma:E\times A \to \R^{n\times n}$ (the \define{coefficients}),
  \item a Borel measurable function $g:E\to [-\infty,\infty]$ (the
  \define{objective function}).
\end{enumerate}

We remind the reader that $\cep$ denotes the set of all
continuous trajectories with values in $E$ that get absorbed once they hit
the boundary $\partial \sO$.

\subsubsection{Weak solutions to controlled SDEs}
With Einstein's convention of summation
over repeated indices used throughout,
we start by making precise what we mean by a controlled diffusion.
\begin{definition}[Weak solutions to controlled SDEs]
\label{def:cont-dif}
A probability measure $\mu$ on $\cep$ is said to be a \define{weak
solution of the controlled SDE}
 \begin{align}
 \label{equ:SDE}
    d \xi^i_t = \beta^i(\xi_t,\alpha_t)\, dt + \sigma^{i}_k (\xi_t,\alpha_t)\,
   dW^k_t,\ \xi_0=x,
 \end{align}
\define{with absorption in $\partial\sO$}  - denoted by
 $\mu\in\sL^x(\beta,\sigma)$ -
 if there exists filtered probability space $(\Omega,\sF,\prfi{\sF_t},\PP)$
 on which three stochastic process $\prfi{W_t}$, $\prfi{\xi_t}$ and
 $\prfi{\alpha_t}$ are defined, such that:
   \begin{enumerate}
     \item $W$ is an $\R^n$ valued $\prfi{\sF_t}$-Brownian motion,
     \item $\xi$ is adapted and $\xi(\omega)\in\cep$ for all $\omega$,
     \item $\alpha$ is $A$-valued and progressively measurable,
     \item $\int_0^t \abs{\beta^i(\xi_u,\alpha_u)}\, du + \int_0^t
(\sigma^i_k(\xi_u,\alpha_u))^2\, du < \infty$, a.s.
for all $i,k$ and  $t\geq 0$,
  \item $\xi_{t} = x+\int_0^t \beta^i(\xi_u, \alpha_u)\, du + \int_0^t
  \sigma^i_k(\xi_u, \alpha_u)\, dW^k_t$, a.s.,  for all $t \in [0,\tau_{\partial\sO}]$,
  where \[ \tau_{\partial\sO} = \inf\sets{t\geq 0}{\xi_t\in\partial\sO}, \eand \]
  \item $\mu$ is the law of $\xi_{\cdot}$ on $\cep$.
 \end{enumerate}
 \end{definition}

\subsubsection{The stochastic optimal control problem}
Given $x\in E$ and $\mu\in\sL^x(\beta,\sigma)$, we set
\begin{align}
\label{equ:J}
J(\mu) = \EE^{\mu}[ G(\xi)] \ewhere
G(\xi) = \liminf_{t\to\infty} g(\xi_t),
\end{align}
with $\xi$ denoting the coordinate map on $\cep$,
where we assume that $g$ is such that
$\EE^{\mu}[ G^+(\xi)]<\infty$ for all $\mu \in
\cup_{x\in E} \sL^x(\beta,\sigma)$.
The \define{value function} of the associated
control problem is then given by
\begin{align}
\label{equ:v}
v(x) = \sup_{\mu\in\sL^x(\beta,\sigma)} J(\mu), \ x\in E.
\end{align}
\begin{remark}
\label{rem:time-dependent}
By choosing the state process $\xi$ appropriately,
this setup includes various common formulations of optimal stochastic
control, including problems on a finite horizon (when $E=E_0 \times [0,T]$
and the last component plays the role of time) with terminal and/or running
costs, discounted problems and stationary problems.
\end{remark}

\subsubsection{DPP for controlled diffusions}
\begin{theorem}[A dynamic programming principle for controlled diffusions
- the weak formulation]
\label{thm:dpp-cont-diff}
Suppose that,
\begin{enumerate}
\item  \label{ite:dcf1}
there exist locally bounded real functions
$\hat{\beta}:E\to\R$ and $\hat{\sigma}:E\to\R$ such that
\[ \sabs{\beta^i(x,\alpha)}
\leq \hat{\beta}(x) \eand \sabs{\sigma^i_k(x,\alpha)} \leq
\hat{\sigma}(x) \eforall \alpha\in A,\]
\item \label{ite:dcf2}
for each
$x\in E$ we have $\sL^x (\beta,\sigma)\ne \emptyset$,  and
\item \label{ite:dcf3} $J(\mu)>-\infty$ for each $\mu\in
\sL^{x}(\beta,\sigma)$.
\end{enumerate}
  Then, the value function $v:E\to (-\infty,\infty]$
 is universally measurable and
satisfies the dynamic programming principle:
\[ v(x) = \sup_{\mu\in\sL^x(\beta,\sigma)} \EE^{\mu}[ v(\xi_{\tau}) \inds{\tau<\infty} +
G(\xi_{\cdot}) \inds{\tau=\infty} ], \eforall x\in E,\]
for each (raw) stopping time $\tau$ on $\cep$.
\label{thm:DPP-diff}
\end{theorem}
\begin{remark}\
\begin{enumerate}
\item
Condition \eqref{ite:dcf1} in Theorem \ref{thm:dpp-cont-diff} is far from
necessary. It is there to ensure existence and is placed mostly for convenience. It can be replaced by a
different condition or relaxed by choosing a different control part
$\Omega^{\alpha}$ of the universal space $\Oax$ in the proof below.
\item A very important feature of our control problem is that the law of
the controlled process depends on the process $\alpha$ only through its
Lebegue-a.e.-equivalence class (as a function of $t$), i.e., it is enough
to think of $\alpha$ as an $\lzer_{A}$-random variable. This feature which
is rarely stressed in the literature, allows us to construct a Polish setup
for the problem, and consequently, prove the DPP.
\end{enumerate}
\end{remark}
\subsection{Proof of Theorem \ref{thm:dpp-cont-diff}}\
\label{sse:prf-dpp-diff}

Our proof of Theorem \ref{thm:DPP-diff} consists of two steps. In the first
one, we observe that the family $\sL^x(\beta,\sigma)$ can be manufactured
by varying admissible controls on a single, universal, filtered
probability space, and that it admits a martingale
characterization there. In the second one we show that this equivalent
setup fits our abstract framework of Section
\ref{sec:mart-gen} so that Theorem
\ref{thm:DPP-mart} can be applied.

\subsubsection{Construction of a universal setup}
\label{sss:const}
Let $\Omega^{\alpha}=\lzer_A$ be the space of all
Lebesgue-a.e equivalence classes of $A$-valued Borel functions from
$[0,\infty)$ to $A$,
and let $\Omega^{\xi}$ be the subspace $\cep$ of the
canonical space $C_{\R^n}$.
Both can be given the structure of a filtered
measurable space, namely $\filta$, $\filtx$,
as described in more detail in subsection \ref{sss:exam-T}
and in Example \ref{exa:subspaces}.
We define the (universal) filtered measurable space
$\filtax$ simply as their product.
In particular $\sF^{\alpha\xi}_t = \sF^{\alpha}_t \otimes \sF^{\xi}_t$.
It will be used in the second step that $\Oax$ is, in fact,
a T-space - the product of T-spaces $\Omega^{\alpha}$ and $\Omega^{\xi}$.

Let $\Coord = \sets{ x_i, x_i x_j }{1\leq i,j \leq n}$ be the family of
coordinate functions and their products on $\R^n$, and let $\QCoord$ denote
an arbitrary, but fixed throughout, countable family of
bounded $C^2$-functions on $\R^n$ such that for each $f\in \Coord$ and each compact set
$K\subseteq \R^n$ there exists $\tf\in \QCoord$ such that $f=\tf$ on $K$.
Also, for $f\in C^2$ and $a\in A$ we define the $\sGa f$ by
    \[ (\sGa f)(x) = \beta^i(x,a) \partial_i f(x) +
    \tot \gamma^{ij}(a,x) \partial_{ij} f(x), \ewith \gamma^{ij} = \tsum_k
    \sigma^i_k \sigma^j_k, \]
\begin{proposition}[A martingale characterization of weak solutions to
controlled SDEs]
\label{pro:mart-form}
The following two statements are equivalent for a
probability measure $\mu$ on $\cep$:
\begin{enumerate}
  \item \label{ite:isweak} $\mu$ is a weak solution to the controlled SDE \eqref{equ:SDE} with
  absorption at $\partial \sO$ starting at $x$, and
  \item \label{ite:mart} there exists a probability measure $\bar{\mu}$ on $\Oax$
  whose $\Omega^{\xi}$-marginal is $\mu$ such that
  \begin{enumerate}
    \item \label{ite:a} $\xi_0=x$, $\bar{\mu}$-a.s.,
    \item \label{ite:b}
       $\int_0^t \abs{\beta^i(\xi_u,\alpha_u)}\, du + \int_0^t
(\sigma^i_k(\xi_u,\alpha_u))^2\, du < \infty$ for all $i,k$ and $t\in
[0,\tau_{\partial \sO}]$, $\bar{\mu}$-a.s., and
      \item  \label{ite:c}
    for each $f\in \QCoord$,
      $f(\xi_t) - f(\xi_0) - \int_0^{t\wedge \tau_{\partial \sO}}
      \sG^{\alpha_u} f
      (\xi_u)\, du$
      is an $(\prfi{\sF^{\alpha\xi}_t},\bar{\mu})$-local martingale.
\end{enumerate}
\end{enumerate}
If \eqref{ite:isweak} holds, then \eqref{ite:c} is true for all $f\in C^2(E)$.
\end{proposition}
The proof follows, almost verbatim, the steps in the standard
proof of the equivalence in the non-controlled case
(see, e.g., Proposition 4.6, p.~315, \cite{KarShr91}) so we omit the
details. The only observation
that needs to be made is that $\alpha$ is not a stochastic process in the
classical sense. This difficulty can be circumvented by considering
appropriate versions as in the following lemma. We remind the reader that an $A$-valued process
$\prfi{\hat{\alpha}_t}$
is considered progressively measurable if $\prfi{\phi(\hat{\alpha}_t)}$ is
progressively measurable for each Borel measurable $\phi:A\to [-1,1]$.

\begin{lemma}
There exists an $\prfi{\sF^{\alpha\xi}_t}$-progressively measurable process
$\prfi{\hat{\alpha}_t}$ with values in $A$ such that
$\{\hat{\alpha}_t(\omega)\}_{t\geq 0}$ is a Leb-a.e.-representative of the
coordinate map $\alpha(\omega)$ for each $\omega$.

Conversely, let $(\xi,\alpha)$ be a pair
consisting of a continuous process $\xi$
with values in $\R^n$
and an $A$-valued progressive process $\alpha$ defined on some  filtered probability
space $(\Omega,\prfi{\sF_t},\FF,\PP)$. Then $(\xi,\alpha)$  admits an
$\Oax$-distribution, i.e., a probability measure $\bar{\mu}$
on $\Oax$ such that the $\PP$-distribution of
$\int_0^t \vp(u,\xi_u,\alpha_u)\, du$  coincides with the
$\bar{\mu}$-distribution of  $\int_{[0,t]} \vp(u,\alpha,\xi)\, d\ld$, for
each bounded and measurable $\vp$ and all $t\geq 0$.
\end{lemma}
\begin{proof}
Let $\phi$ be an isomorphism (a bimeasurable bijection) between $A$ and the closed
interval $[-1,1]$. Given $\alpha(\omega) \in \lzer([0,\infty),A)$, we define
$\hat{\alpha}$ by
\[ \hat{\alpha}(t) = \phi^{-1}\Big(\liminf_{n\to\infty} \Phi^n_t(\omega)
\Big) \ewhere
\Phi^n_t(\omega)=   \oo{n}
\int_{(t-1/n)^+}^t \phi\big(\alpha_u(\omega)\big)\, du.\]
It is straightforward to check that $\hat{\alpha}(\omega)$
is a representative of $\alpha(\omega)$ for each $\omega$. Moreover
$\phi(\hat{\alpha})$ (and, therefore, $\alpha$) is
a progressively-measurable process, as a pointwise limit of continuous
adapted processes.

For the converse, and
under the assumptions of the second part of the Lemma, let $\bar{\mu}$ be
the pushforward of $\PP$ via the map $\Phi:\Omega \to \Oax$
defined as follows:
\[ \Phi(\omega) = \Big( (\xi_t(\omega))_{t\geq 0}, \alpha(\omega)  \Big),\]
where $\alpha(\oemga)$ is the Leb-a.e.-equivalence class of
$(\alpha_t(\omega))_{t\geq 0}$. (Progressive) measurability of $\alpha$ guarantees
that $\Phi$ is a measurable map. The equality of the distributions of two
integrals in the statement is then a simple consequence of the
monotone-class theorem.
\end{proof}

\subsubsection{An application of the abstract DPP}
\label{sss:abst-visc}
Proposition \ref{pro:mart-form} allows us to reformulate our control
problem so as to fit the setting of the first part of our paper. Indeed, it
states that the value function $v(x)$ can be represented as
\[ v(x) = \sup_{\bar{\mu}\in\bar{\sP}(x)} \EE^{\mu}[ G(\xi) ]\]
where $\bar{\sP}(x)$ is the family of all probability measures
on $\Oax$ such that \eqref{ite:a}, \eqref{ite:b} and
\eqref{ite:c} hold, and our job is to show that it is, in fact, a
martingale generated control correspondence which satisfies all the requirements
of the abstract Theorem \ref{thm:DPP-mart}.

Thanks to the discussion and examples in subsections \ref{sss:categ} and
\ref{sss:exam-TC}, the space $\Oax$ admits a natural structure of a
TC-space, with the strict concatenation used for the $\xi$ component. The
map $X:\Oax\to E$, given by $X(\xi,\alpha) = \liminf_{t\to\infty} \xi_t$
computed componentwise, and
suitably measurably altered to take values in $E$ and when the limits
inferior take infinite values, so that $X_t(\xi,\alpha) = \xi_t$.
Given that the concatenation operator in $\alpha$ requires no
compatibility conditions, and the one in $\xi$ is strict, the product
concatenation operator $\ast$ factors through $X$ (and is a factor of $X$).
Also, there is a naturally-defined shift operator $\theta$ on $\Oax$.

Condition \eqref{ite:dcf1} of Theorem \ref{thm:dpp-cont-diff} takes
care of the integrability condition \eqref{ite:b} of Proposition
\ref{pro:mart-form}, so we can conclude that
we are, indeed,  dealing with a martingale-generated control correspondence
with the state map $X$, generated by the family $\sD$ which consists of
(well-defined) maps
of the form
\[ F( \alpha,\xi )_t =
f(\xi_t) - f(\xi_0) - \int_0^{t\wedge \tau_{\partial \sO}}
\sG^{\alpha_u}f(\xi_u)\, du\]
with $f$ ranging through the countable set $\QCoord$.
The last thing we need to check, before we can apply Theorem
\ref{thm:DPP-mart}, is that each such $F$ is a TC-morphism into
$(\dr^0,\star)$. We fix $f\in \QCoord$, and note that the corresponding
functional $F$ clearly takes values in $\dr^0$ and that it is
non-anticipating. To establish the TC-morphism property let us fix
$s,t\in\Time$ and $\omega,\omega'\in\Oax$ such that
$\omega$ is compatible with $\omega'$ at $t$.
The case of $s\leq t$ is straightforward, so suppose $s>t$.
Since the $\xi$ component uses the strict concatenation operator,
then $\xi_t(\omega)=\xi_0(\omega')$, and furthermore:
\begin{align*}
  \tau_{\partial\sO}(\omega)\leq t
  &\quad\Leftrightarrow\quad
  \xi_t(\omega)\in\partial\sO
  \quad\Leftrightarrow\quad
  \xi_0(\omega')\in\partial\sO
  \quad\Leftrightarrow\quad
  \tau_{\partial\sO}(\omega')=0
\end{align*}
Combining this with the properties of concatenation gives:
\begin{align*}
\int_{t\wedge \tau_{\partial \sO}}^{s\wedge\tau_{\partial \sO}}
\sG^{\alpha_u}f(\xi_u(\omega\ast_t\omega'))\, du
&= \inds{\tau_{\partial\sO}>t}(\omega)
\int_{t\wedge \tau_{\partial \sO}}^{s\wedge\tau_{\partial \sO}}
\sG^{\alpha_u}f(\xi_u(\omega\ast_t\omega'))\, du \\
&= \inds{\tau_{\partial\sO}>0}(\omega')
\int_{0}^{(s-t)\wedge\tau_{\partial \sO}}
\sG^{\alpha_u}f(\xi_u(\omega'))\, du \\
&=
\int_{0}^{(s-t)\wedge\tau_{\partial \sO}}
\sG^{\alpha_u}f(\xi_u(\omega'))\, du
\end{align*}
Putting everything together gives:
\begin{align*}
F(\omega\ast_t \omega')_s
&= f(\xi_s(\omega \ast_t \omega')) - f(\xi_0(\omega\ast_t\omega'))
- \int_0^{s\wedge \tau_{\partial \sO}}
\sG^{\alpha_u}f(\xi_u(\omega\ast_t\omega'))\, du \\
&= \left(
f(\xi_t(\omega*_t\omega')) - f(\xi_0(\omega\ast_t\omega'))
- \int_0^{t\wedge \tau_{\partial \sO}}
\sG^{\alpha_u}f(\xi_u(\omega\ast_t\omega'))\, du
\right) \\
&\quad + \left(
f(\xi_s(\omega \ast_t \omega')) - f(\xi_t(\omega*_t\omega'))
- \int_{t\wedge \tau_{\partial \sO}}^{s\wedge\tau_{\partial \sO}}
\sG^{\alpha_u}f(\xi_u(\omega\ast_t\omega'))\, du \right) \\
&= \left(
f(\xi_t(\omega)) - f(\xi_0(\omega))
- \int_0^{t\wedge \tau_{\partial \sO}}
\sG^{\alpha_u}f(\xi_u(\omega))\, du
\right) \\
&\quad + \left(
f(\xi_{s-t}(\omega')) - f(\xi_0(\omega'))
- \int_{0}^{(s-t)\wedge\tau_{\partial \sO}}
\sG^{\alpha_u}f(\xi_u(\omega'))\, du  \right) \\
&= F(\omega)_t + F(\omega')_{t-s}
= (F(\omega)\star_t F(\omega'))_s
\end{align*}

\subsection{Viscosity solutions}\

We conclude this example by showing how our result can be applied to
show that value functions of stochastic control problems are viscosity
solutions to the associated Hamilton-Jacobi-Bellman equations under weak
conditions. In particular, we do not require that the equation itself
admit an a-priori solution, or that any solution is smooth or unique (i.e,
that the comparison principle hold). Our results, in particular, imply some of
the results in \cite{BaySir13}, \cite{BouTou11} and the follow-up papers under weaker
assumptions.
We note that the lack of any strong ellipticity allow us
keep assuming, without loss of generality,
that the problem is time-independent; time can be
incorporated as
just another (space) variable with linear dynamics and the terminal
condition imposed as part of the boundary condition.

For a $C^{2}$ function $\vp:\sO\to\R$ we define the \define{Hamiltonian}
$H \vp: \sO\to (-\infty,\infty]$ by
\[ H \vp (x) = \sup_{a \in A} \sGa \vp (x)
    =  \sup_{a\in A} \Big(  \beta^i(x,a)
    \partial_{x_i} \vp(x) +
    \tot \gamma^{ij}(x,a) \partial_{x_i x_j} \vp(x) \Big).\]
\subsubsection{The viscosity property of the value function}

\begin{definition}
\label{def:viscosity}
Let $v$ be a real-valued function defined in a neighborhood $\sV$ of a point
$\bar{x}\in \sO$, and let
$v_*$ and $v^*$ denote its lower and upper semicontinuous envelopes,
respectively. We say that $v$  is a
\begin{enumerate}
\item
\define{viscosity supersolution} of the equation $Hv =0$ at $\bar{x}$ if
 $H\vp (\bar{x})\leq 0$
for each $\vp\in C^{2}(\sV)$ with the property that
$\vp(\bar{x})=v_*(\bar{x})$ and $\vp(x)<v_*(x)$ for
$x\in\sV\setminus\set{\bar{x}}$ , and
\item
\define{viscosity subsolution} of the equation $Hv =0$ at $\bar{x}$ if
 $H\vp (\bar{x})\leq 0$
for each $\vp\in C^{2}(\sV)$ with the property that
$\vp(\bar{x})=v^*(\bar{x})$ and $\vp(x)>v^*(x)$ for
$x\in\sV\setminus\set{\bar{x}}$ .
\end{enumerate}
A function which is both a viscosity supersolution and a viscosity
subsolution is called a \define{viscosity solution} to $H v=0$ at
$\bar{x}$.
\end{definition}
For $x\in\R^n$ and $r>0$
we define
\[ \tau^{r, x}=\inf\sets{t\geq 0}{ d(x,\xi_t) \geq r}\wedge r,\]
where $d$ denotes the Euclidean distance on $\R^n$, so that
$\tau^{r,x}$ is a raw stopping times on $\Oax$.
\begin{theorem}
\label{thm:viscosity}
Given $\bar{x}\in\sO$, suppose that there exists a neighborhood
$\sV$ of $\bar{x}$ in $\sO$ such that
\begin{enumerate}
\item \textbf{(availability of DPP)} the assumptions of Theorem \ref{thm:dpp-cont-diff} hold and $v$ is
finite on $\sV$,
\item \textbf{(continuity of coefficients)} $x\mapsto \beta^i(x,a)$ and $x\mapsto \sigma^i_k(x,a)$
are continuous functions on $\sV$ for all $a\in A$,
\item \label{ite:loc-const}\textbf{(admissibility of locally constant controls)} there exists a constant $r>0$ such that for each $x\in \sV$ and
$a\in A$ there exists
a control process $\prfi{\alpha_t}$ and an associated weak solution
$\prfi{\xi_t}$ of the controlled SDE \eqref{equ:SDE} with $\xi_0=x$
(defined on some filtered probability space) such that
\[ \alpha_t = a\  \efor t\in [0,\tau]\text{ a.s., } \ewhere \tau=\inf\sets{t\geq
0}{d(\xi_t,\bar{x})\geq r}\wedge r.\]
\end{enumerate}
Then  the value function $v$ is a viscosity solution to $Hv =0$ at
$x_0$.
\end{theorem}
\begin{proof}
We split the proof into two parts, in which we establish the supersolution
and the subsolution property of $v$ separately.

\emph{The supersolution property.}
We take $\vp\in C^2$ which touches $v_*$ at $\bar{x}$ from below, i.e.
$v_*(\bar{x}) = \vp(\bar{x})$ and $\vp(x)<v_*(x)$ for $x\ne \bar{x}$
This implies that
  there exists a sequence $\seqm{x}$ such that
 \begin{align}
 \label{equ:xm}
   v(x_m) \leq
    \vp(x_m) + \oo{m} \eand
   d(x_m,\bar{x})\leq \oo{m}.
 \end{align}
Suppose, for contradiction,  that $  H\vp (\bar{x})>0$. Then there exists
$a\in A$ such that $(\sG^a \vp)(\bar{x})>0$. Since $\sG^a \vp$ is
continuous in $x$, there exist constants $\eps>0$ and $r>0$ such that
$(\sG^a \vp)(x)\geq \eps $ when $d(x,\bar{x})\leq r$.
Using the fact that
$\vp(x)<v_*(x)$ as soon as $x\ne \bar{x}$ and that the function $v_*-\vp$
is lower semicontinuous, we find that
\[ \delta = \min\sets{v_*(x) - \vp(x)}{ d(x,\bar{x})=r }>0.\]
For each $m\in\N$, let $\mu_m$ be the law of the weak solution
$\prfi{\xi_t}$ described in part \ref{ite:loc-const} of the statement,
where we assume, without loss of generality, that the same constant $r>0$,
as above, can be used.
Proposition \ref{pro:mart-form}
and the local nonnegativity of $\sG^a \vp-\eps$ imply that
$\vp(\xi_t) - \eps t$ is a bounded
$\mu_m$-submartingale under $\mu_m$ on $[0,\tau^{r,\bar{x}}]$. Therefore, with
$\tau=\tau^{r,\bar{x}}$ and for $m>1/r$, we get
\begin{align*}
\vp(x_m) &\leq
\EE^{\mu_m}[ \vp(\xi_{\tau}) - \eps \tau ] \leq
\EE^{\mu_m}[ \vp(\xi_{\tau})\inds{\tau<r}] +
\EE^{\mu_m}[ (\vp(\xi_{\tau})-\eps r) \inds{\tau=r}] \\
&\leq \EE^{\mu_m}[ (v_*(\xi_{\tau}) - \delta) \inds{\tau<r} ] +
\EE^{\mu_m}[ (v_*(\xi_{\tau}) - \eps r) \inds{\tau=r}]\\ & \leq \EE^{\mu_m}[
v_*(\xi_{\tau}) ] - \min(\delta,\eps r).
\end{align*}
Using the dynamic programming principle of Theorem \ref{thm:dpp-cont-diff}
and the relation \eqref{equ:xm} above, we finally obtain
\begin{align*}
  v(x_m) - \oo{m} + \min(\delta,\eps r) \leq \EE^{\mu_m}[ v_*(\xi_{\tau}) ]
  \leq \EE^{\mu_m}[ v(\xi_{\tau}) ] \leq \sup_{\mu\in
  \sL^{x_m}(\beta,\sigma)}
  \EE^{\mu}[ v(\xi_{\tau}) ] = v(x_m),
\end{align*}
and reach a contradiction by taking $m$ large enough.

\emph{The subsolution property.} We pick $\vp\in C^2$ which touches $v^*$ at $\bar{x}$ from above, i.e.
$v^*(\bar{x}) = \vp(\bar{x})$ and $\vp(x)>v_*(x)$ for $x\ne \bar{x}$.
As in the first part of the proof, this implies that
  there exists a sequence $\seqm{x}$ such that
 \begin{align}
 \label{equ:xm2}
   v(x_m) \geq
    \vp(x_m) - \oo{m} \eand
   d(x_m,\bar{x})\leq \oo{m}.
 \end{align}
Suppose, for contradiction, that $ H\vp (\bar{x})<0$. Being representable
as a supremum of continuous functions, $H\vp$ is upper semicontinuous, and
so there exist constants $r>0$ and $\eps>0$ such that $ H\vp (x) \leq -\eps$
for all $x$ with $d(x,\bar{x})\leq r$.
Using the fact that
$\vp(x)>v^*(x)$ as soon as $x\ne \bar{x}$ and that the function $\vp - v^*$
is lower semicontinuous, we find, as above, that
\[ \delta = \min \sets{\vp(x) - v^*(x) }{ d(x,\bar{x})=r }>0.\]
Let the laws $(\mu_m)_{m\in\N}$
be defined as in the first part of the proof, so that
  under each $\mu_m$ the process $\vp(\xi_t) + \eps t$ is
  supermartingale on $[0,\tau^{r,\bar{x}}]$. It follows that, with
  $\tau=\tau^{r,x}$, we have
\begin{align*}
     \vp(x_m) &\geq
     \EE^{\mu}[ \vp(\xi_{\tau}) + \eps \tau ] =
     \EE^{\mu}[ (\vp(\xi_{\tau}) + \eps \tau) \inds{\tau=r}] +
     \EE^{\mu}[ (\vp(\xi_{\tau}) + \eps \tau) \inds{\tau<r}]\\
     &\geq \EE^{\mu}[ (v^*(\xi_{\tau}) + \delta) \inds{\tau=r}] +
     \EE^{\mu}[ (\vp(\xi_{\tau}) + \eps r) \inds{\tau<r}]
     \geq \EE^{\mu}[ v(\xi_{\tau}) ]  + \min(\delta,\eps r)
   \end{align*}
We take a supremum over all $\mu\in\sP^{x_m}$ on the right hand side and
use the DPP to conclude that
$\vp(x_m) \geq v(x_m) + \min(\delta,\eps r)$ for all $m$ - a contradiction
with \eqref{equ:xm2}.
\end{proof}

\section{Application 2 - Singular Control Problems}
\subsection{The Monotone-Follower Problem}
We show how singular control problems fit our framework on the
example of the celebrated Monotone Follower Problem (first formulated by
Bather and Chernoff \cite{BatChe67}, analyzed rigorously by Karatzas
and Shreve in \cite{KarShr84} and studied in many papers since).
Formally, the Monotone Follower Problem asks for a minimal cost incurred
while controlling a Brownian motion $W$ by adding to it a non-decreasing
  left-continuous
  process $\alpha$. The cost is typically given by
$$ \EE\left[
	 \int_0^\tau f(t) d\alpha_t
	+ g(W_T-\alpha_T)
+ \int_0^\tau h(t,W_t-\alpha_t) dt
        \right] , $$
where $g$ and $h$ model the deviation of the controlled trajectory
$W+\alpha$ from the desired optimal position and $f$ plays the role of
``fuel'' cost.
\subsection{Formulation in our framework} To make it easier to focus on the
issues pertinent to the proof of the DPP, we generalize the problem to a
degree. The continuous variables, such as time, running cost or the
Brownian motion from the above description will be replaced by a general,
multidimensional diffusion. This will not only allow us to reuse many of the
conclusion of the previous section, but also to get a clearer understanding
of the role different parts play as far as DPP is concerned.

\subsubsection{The space $\Omega$.}
Given $m,n\in\N$, let $\sO \subseteq \R^{m+n}$ be a nonempty open set
with closure $E=\Cl\sO$, which will play the role of
our state space. Let $\Cr$ and $G_{\R}$ denote the canonical spaces
of all continuous and  c\` agl\` ad paths, respectively,  with values in
$\R$, and let $\Grup$ denote the subset
of $G_{\R}$ consisting of nondecreasing paths.
Let $\Omega^{X}$ denote the space of paths in
$\Cr^m \times (\Grup)^n $ with values in $E$, absorbed upon entry in  $\partial\sO$, i.e. stopped
at the canonical stopping time
\[ \tau_{\partial \sO}(\omega) = \inf \sets{t\geq 0}{ \omega(t) \in \partial
\sO} \efor \omega \in \Omega^X.\]
With the  control component taking value
in $\Omega^{\alpha}=\Grup$, the space $\Omega$ is defined as
the subset of $\Omega^{X} \times \Omega^{\alpha}$ consisting of those paths
$(X,\alpha)$ stopped once $X$ hits $\partial \sO$. Equivalently, $\Omega$
is the set of paths in $\Omega^X \times \Omega^{\alpha}$ that get absorbed
once the coordinate map $(X,\alpha)$ enters the set $\partial \sO \times
\R$. We overload the notation $\tau_{\partial \sO}$ to denote the hitting
time of $\partial \sO \times \R$, when considered as a stopping time on
$\Omega$.

The first $n$ coordinate maps on $\Omega$ (corresponding to continuous
paths) are denoted by $Y$, the next $m$ (corresponding to left-continuous
paths) by $Z$ and the last one by $\alpha$, so that $\omega(t) =
(Y_t(\omega), Z_t(\omega), \alpha_t(\omega))$, for $\omega\in \Omega$ and
$t\geq 0$.

\subsubsection{The T-space, TC-space structures}
We use the standard truncations on each of the components of $\Omega$.
To see that $\Omega$ carries a natural structure of a T-space, we simply
need to combine the discussion in paragraph \ref{par:G} in  subsection
\ref{sss:exam-T} with the product construction of paragraph
\ref{par:prod}. It can be upgraded to a TC-space by equipping it with
\begin{enumerate}
\item the
strict concatenation operator $\bullet$, as  defined in equation
\eqref{equ:str-form} in subsection \ref{par:strict-conc}, on $\Omega^X$
(i.e., for the first $m+n$ coordinates), and
\item the adjusted concatenation
$\star$, as defined by \eqref{equ:adj-concat},
on $\Omega^{\alpha}=\Grup$ (for the last coordinate).
\end{enumerate}
The so-obtained concatenation on $\Omega$ will be denoted by $\ast =
(\bullet,\star)$.

\subsubsection{The state $X$ and the cost functional $G$.}
Let
\[ \ulim: \Omega^X \to E \]
be a "Banach limit", i.e., a map with the following properties:
\begin{enumerate}
\item Its value on the trajectory $\omega$
coincides with the pointwise limit $\lim_{t\to\infty} \omega(t)$
whenever this limit exists; in particular, it equals the value at which
$\omega$ is absorbed, when absorption happens.
\item It returns a value in $E$ in a Borel measurable way.
\item It is invariant under the action of the shift operator.
\end{enumerate}
A fairly general construction of such a map on spaces of right-continuous
trajectories can be found in \cite[Lemma 3.12, p.~1614]{Zit14}. A closer
inspection of the proof reveals that the right-continuity assumption can be
replaced by the assumption of left continuity, and that the conclusion of
the theorem applies to the present setting. Given such a map $\ulim$, we
simply define
\[ X(\omega) = \ulim \omega^X \efor \omega= (\omega^X,\omega^{\alpha})
\in \Omega.\]
In agreement with the definition of the coordinate maps $Y_t$ and $Z_t$ above,
we split the first $n$ and the last $m$ coordinates of $X$ into $Y$ and
$Z$, i.e. $X(\omega) = (Y(\omega), Z(\omega))$. This way, since we
are working with the standard truncation, we have
\[ X_t(\omega) =
X(\omega_{\leq t}) = ( Y(\omega_{\leq t}), Z(\omega_{\leq t})) =
(Y_t(\omega), Z_t(\omega)).\]

With $X$ defined, the cost function $G$ is simply a Borel function of $X$:
  \begin{align}
  \label{equ:G-def}
    G(\omega) = g(X(\omega)),
  \end{align}
where we assume throughout that $g$ is nonnegative so as not to need to pay
attention to
integrability conditions in the sequel. Much less restrictive assumptions are
also possible.

\subsubsection{The control correspondence $\sP$}
The control correspondence describing our monotone-follower problem will naturally factor through the state map $X$, so
we define the family $(\osP^x)_{x\in E}$, and use it to construct the
control correspondence in the usual way $\sP(\omega) = \osP^{X(\omega)}$.
Heuristically, the dynamics of the state $X_t=(Y_t,Z_t)$ under
$\osP^x$ for $x\in E$ can be described as follows: $Y$ is a diffusion on
$\R^n$, with coefficients depending on $X_t$,
absorbed once $X_t$ hits $\partial \sO$. The left-continuous
component $Z$ ``moves'' as follows
 \begin{align}
 \label{equ:dZ}
    dZ_t = c(Y_t)\, d\alpha_t,
 \end{align}
 where $c$ is a vector of $m$ nonnegative and continuous functions.

To simplify the exposition, we express $\osP$ as an
intersection of two control correspondences $\osP_c$ and $\osP_l$,  where $\osP_c$
``constrains'' the motion of continuous portion $Y$
and $\osP_l$ the left-continuous portion $Z$, of the state process $X$.
To define $\osP_c$ we follow the approach of section \ref{sec:app1} and
consider a family $\sD_c$ of maps from $\Omega$ to $\Cr \subseteq D_{\R}$
given by
\[ F^f( \omega)_t =
f(Y_t(\omega)) - f(Y_0(\omega)) - \int_0^{t\wedge \tau_{\partial
\sO}(\omega)}
\sG^{Z_u(\omega)} f(Y_u(\omega))\, du,\]
where $f$ ranges through the set $\QCoord$ as in the second paragraph of
section \ref{sss:const}, and $\sG^{z}$ is a differential operator of the
form
\[ (\sG^z f)(y) = \beta^i(y,z) \partial_i f(y) +
\tot \Sigma^{ij}(y,z) \partial_{ij} f(y), \ewith \Sigma^{ij} = \tsum_k
\sigma^i_k \sigma^j_k,\]
with coefficients $\beta$ and $\sigma$ measurable, locally bounded and globally
Lipschitz in $y$. These conditions are imposed to ensure that the control correspondence
$\osP_c$ generated by $(\sD,X)$ is well-defined and non-empty.

We note here that the dependence of any $F$ on
$\alpha$ is trivial; that means that even though we think of $\alpha$ as a
control, its influence on $F$ factors entirely through the left-continuous
process $Z$ and does not show up in $\osP_c$. To describe how $Z$ depends on $\alpha$, we need to introduce the
control correspondence $\osP_l$.
To describe it rigorously, we first need to agree on how to define the integral
with respect to a left-continuous process
in \eqref{equ:dZ} above. Such a construction has been carried out already
 in \cite[Remark 5.3., p.~873]{KarShr84}; we simply exhibit parts of their
 discussion for the convenience of the reader.
Given a nondecreasing \cg function
$\alpha:[0,\infty) \to \R$, we define the \cd function $\alpha^+:[0,\infty)\to \R$ by setting
$\alpha^+_t:=\alpha_{t+}=\inf_{u>t } \alpha_u$.
For a locally bounded Borel function $\gamma:[0,\infty)\to \R$, we define
   \begin{align}
     \int_0^t \gamma(u)\, d\alpha_u := \begin{cases}
    0, & t=0, \\
     \gamma(0) \Delta \alpha_0 +\int_{(0,t)} \gamma(u)\,
     d\alpha^+_u, & t>0,\end{cases}
     \label{equ:lc-integral}
   \end{align}
   where $\Delta\alpha_0 = \alpha^+_{0}-\alpha_0$ and
the integral on the right-hand side is the Lebesgue-Stieltjes
integral with respect to the measure induced by $\alpha^+$ on $(0,t)$. We
immediately observe that the function $\zeta_t = \int_0^t \gamma(u)\, d\alpha_u $ is
\cg and satisfies $\zeta^+_0  =\gamma(0) \Delta \alpha_0$. We also record, for later
use, the following characterization:
\begin{lemma}
\label{lem:left-integral}
Suppose that
$\alpha \in \Grup$ and
that $\gamma:[0,\infty)\to\R$ is continuous. For $\zeta\in \Grup$,
the following two conditions are equivalent
\begin{enumerate}
  \item
$\zeta =\zeta_0+\int_0^{\cdot} \gamma(u)\, d\alpha_u$, and
 \item
 $\Delta \zeta_{0+} = \gamma(0) \Delta \alpha_{0+}$ and for all rational $0<r<s$ and each
 $n\in\N$ there exist rationals $p,q\in (r,s)$ such that
  \begin{align}\label{equ:qu}
   (\gamma(p)-\oo{n})  (\alpha^+_{s}-\alpha^+_{r}) \leq \zeta^+_r - \zeta^+_s \leq
  (\gamma(q)+\oo{n})(\alpha^+_{s}-\alpha^+_{r}).
  \end{align}
  \end{enumerate}
\end{lemma}
\begin{proof}
Thanks to right continuity of $\zeta^+$,
(1) above is equivalent to
 $\Delta \zeta_{0+} = \gamma(0) \Delta \alpha_{0+}$ and
  \begin{align}
  \label{equ:zeta-plus}
    \zeta^+ - \zeta^+_0 = \int_{(0,\cdot]} \gamma(u)\, d\alpha^+_u.
  \end{align}
Using the right continuity of $\alpha^+$ and the continuity of $\gamma$
(which guarantees the equivalence between the Riemann-Stieltjes and the
Lebesgue-Stieltjes integration in this case) we conclude that
the equality in \eqref{equ:zeta-plus} is equivalent to
\[  \forall\, u<v\in (0,\infty),\  \left( \inf_{t\in [u,v]} \gamma(t)
\right)
(\alpha^+_u-\alpha^+_v) \leq \zeta^+_u - \zeta^+_v \leq
  \left( \sup_{t\in [u,v]} \gamma (t)\right)(\alpha^+_u-\alpha^+_v).\]
Thanks to the right continuity of $\alpha^+$ and $\zeta^+$, this is easily
seen to be equivalent to (second statement in ) (2) above.
\end{proof}
Given a continuous function $c: \R^m \to \R^n$, for $x\in E$ we define
\begin{align*}
    \osP_l(x) = \{ \mu \in \prob(\Omega)\,:\,  X_0=x,\, Z = Z_0+
    \int_0^{\cdot} c(Y_u)\, d\alpha_u, \mu-\text{a.s.}\}
 \end{align*}
 where the left-continuous integral is interpreted component-wise.
We set $\osP(x)= \osP_c(x) \cap \osP_l(x)$ and define the
\define{value function} of the associated
control problem by
\begin{align}
\label{equ:v2}
v(x) = \inf_{\mu\in\osP(x)} \EE^{\mu}[ G ], \ x\in E.
\end{align}
\begin{remark}
\label{rem:class} To see how the classical monotone-follower fits into this
framework, we take $Y=(T,W,H)$ and $Z=(L,C)$, where, informally, the
components have the following dynamics:
\begin{align*}
 d T_t &= - dt, && \text{ time-to-go,}\\
 d W_t &= dW_t, && \text{ Brownian motion}\\
 d H_t &= h(-T_t, W_t, L_t)\, dt, && \text{ running cost }\\
 d L_t &= d\alpha_t, && \text{ position of the follower, and } \\
 d C_t &= f(T_t)\, d\alpha_t, && \text{ fuel cost,}
\end{align*}
where $f$ and $h$ are nonnegative and continuous.
The state space $E$ is defined by $E=\Cl \sO$, where
\[ \sO = (-\infty,0) \times \R \times (0,\infty) \times (0,\infty) \times
(0,\infty),\]
so as to keep the components $H$ and $C$ nonnegative. This will also make
sure that the state process will exit $E$ when (and only when) $T_t=0$.
A typical cost functional $G$ will be of the form
$G(X)= H+C+g(W,L)$, where $g$ is a nonnegative Borel function.

\end{remark}
\subsection{The Dynamic Programming Principle} With all the components
of our framework in place, we are ready to prove the following result:
\begin{proposition}[DPP for the monotone-follower problem]
Given the setting described above, the value function $v: E \to (-\infty,\infty]$ is universally measurable
and satisfies the dynamic programming principle
\[ v(x) = \sup_{\mu\in\osP^x} \EE^{\mu}[ v(X_{\tau}) \inds{\tau<\infty} +
G \inds{\tau=\infty} ], \eforall x\in E,\]
for each (raw) stopping time $\tau$ on $\cep$.
\label{pro:DPP-diff}
\end{proposition}
\begin{proof} As in the previous section, we establish three key properties,
namely, analyticity, concatenability and disintegrability, and use Theorem
\ref{thm:DPP}. The additional requirement that $G$ be a tail random
variable follows directly from the fact that it was defined in
\eqref{equ:G-def} using a ``Banach limit'', i.e.,
in a shift-invariant way. The membership in the class
$\slzm(\osP)$ of lower semi-integrable random variables is guaranteed by the
assumption that the function $g$ in \eqref{equ:G-def} is bounded from
below.

\emph{Analyticity:}
To establish the analyticity of $\osP$ it will be enough to show that both $\osP_c$
and $\osP_l$ are analytic (see Remark  \ref{rem:unions}).
  All the maps in $\sD_c$ are clearly non-anticipating  and take values in
  $\Cr \subseteq \dr$, so we can apply
  Proposition \ref{pro:analy-mart} to conclude that $\osP_c$ is analytic.

The analyticity of $\osP_l$, follows from Lemma
\ref{lem:left-integral}. Indeed, it expresses $\osP_l$ as a result
of a countable collection of Borel-preserving operations on cylinders.

\emph{Concatenability:}
Just like in the case of analyticity, Remark \ref{rem:unions}
allows us to
prove concatenability of $\osP$ by proving it separately for $\osP_c$ and
$\osP_l$. Starting with $\osP_c$, we simply note that the maps $F^f$ in
$\sD_c$
are $\Cr$-valued and therefore canonically locally bounded. Their TC-morphism
property is established exactly like in section \ref{sss:abst-visc} above,
so we can use Proposition \ref{pro:concat-mart} to conclude that $\osP_c$ is
closed under concatenation.

Next, we turn to the concatenability of $\osP_l$.
Given $t\geq 0$ let $\omega,\omega'\in \Omega$ be such that
\begin{enumerate}
\item $X_0(\omega') = X_{t}(\omega)$,
\item $C(\omega) = \int_0^{\cdot} c(Y_u(\omega))\, d\alpha_u(\omega)$, and
\item $C(\omega') = \int_0^{\cdot} c(Y_u(\omega'))\, d\alpha_u(\omega')$,
\end{enumerate}
We note that these properties hold for $(\omega,\omega')$
with probability $1$, under $\mu \otimes_t \nu$. Using the fact that $\ast$
is strict in the first $m+n$ components and adjusted in $\alpha$, we observe
that for $s>t$ we have
  \begin{multline*}
    C_s(\oto) - C_{t+}(\oto) =
    C_{s-t}(\omega') - C_{0+}(\omega') =
    \int_{(0,s-t)} c(Y_u(\omega'))\, d\alpha^+_u(\omega') = \\
      = \int_{(t,s)} c(Y_{u-t}(\omega'))\, d\alpha^+_{u-t}(\omega')
      = \int_{(t,s)} c(Y_u(\oto))\, d\alpha^+_u(\oto),
      \end{multline*}
  as well as
  \begin{align*}
    C_{t+}(\oto) &- C_t(\oto) = C_{0+}(\omega') - C_0(\omega') =
    c(Y_0(\omega') ) (\alpha_{0+}(\omega') - \alpha_0(\omega')) \\ &=
    c(Y_t(\oto) ) (\alpha_{t+}(\oto) - \alpha_t(\oto)).
  \end{align*}
These two observations make it straightforward to complete the proof of the
concatenability of $\osP_l$.

\emph{Disintegrability:} While disintegrability cannot be established by
showing it for
$\osP_c$ and $\osP_l$ separately, we can use Proposition \ref{pro:dis}, whose
conditions are easily shown to hold in the present setting, to
perform most of the work for us. Indeed, given $\omega_0\in\Omega$ and
$\mu\in \osP_c(\omega_0)$ and $\kappa\in\Stop$, it states that there exists
a version $x\mapsto \bar{\nu}_x$ of the regular conditional probability
$\mu(\theta_{\kappa}\in \cdot| X_{\kappa}=x)$ with the following two
properties: 1) $\nu \in \sS(\osP_c)$ and 2) $\mu = \mu \ast_{\kappa} \nu$, where
$\nu = \bar{\nu} \circ X$.
In order to complete the proof, we need to show that a version of $\nu$
with $\nu \in \sS(\osP_l)$, can be constructed. Let $A$ denote the set of
all $\omega\in\Omega$ such that $Z(\omega) - Z_0(\omega) = \int_0^{\cdot}
c(Y_u(\omega))\, d\alpha_u(\omega)$. For any $x\in E$ and any
$\mu \in \bar{\sP}(x)$ we have
$\mu( A ) = 1$. Therefore, by the concatenability property established
above, we have
 \begin{align*}
    1 &= \int \ind{A}(\omega) \, \mu(d\omega) =
   \int \int  \ind{A}(\omega\ast_{\kappa} \omega')
   \bar{\nu}_{X_{\kappa}(\omega)}(d\omega')
   \mu(d\omega) \leq
   \int \int  \ind{A}(\omega') \bar{\nu}_x(d\omega') \mu_{X_{\kappa}}(dx),
 \end{align*}
and, so, there exists a $\sN_1 \in \Borel(E)$ with  $\mu_{X_{\kappa}}(\sN_1)=0$
and
such that for $x\in E\setminus \sN_1$ we have $\bar{\nu}_x( A )=1$.
Similarly, $\bar{\nu}_x(X_0=x)=1$ for all $x\in E\setminus \sN_2$, where
$\sN_2$ is a $\mu_{X_{\kappa}}$-null set in $\Borel(E)$.
It remains to redefine $\bar{\nu}$ on $\sN_1\cup\sN_2$ so that $\bar{\nu} \in
\sS(\osP_c)\cap \sS(\osP_l)$. This is easily achieved by picking an arbitrary
selector $\bar{\nu}' \in \sS(\bar{\osP}_c\cap \bar{\osP}_l)$ and setting
setting $\bar{\nu}_x = \bar{\nu}'_x$ for all $x\in \sN_1\cup\sN_2$.
\end{proof}

\end{document}